\documentclass[final,5p,times,twocolumn]{elsarticle}
\usepackage [utf8]{inputenc}
\usepackage{lmodern}
\usepackage{appendix}
\usepackage{amsmath}
\usepackage{graphicx}
\usepackage[colorlinks=true, allcolors=black]{hyperref}
\usepackage{amssymb}
\usepackage{amsfonts}
\usepackage{amsthm}
\usepackage{colortbl}
\usepackage{dsfont}
\usepackage{bm}
\usepackage{algorithm}
\usepackage{algorithmic}
\usepackage{dsfont}
\usepackage{caption}
\usepackage{subcaption}
\usepackage[table, dvipsnames]{xcolor}
\usepackage{tabularx}
\usepackage{lipsum}
\usepackage{mathtools}
\usepackage{listings}
\usepackage{booktabs}

\newcommand{\probP}{\mathbb{P}}
\newcommand{\espE}{\mathbb{E} }
\newcommand{\indic}{ \mathds{1} }
\newtheorem{theorem}{Theorem}[section]
\newtheorem{lemma}[theorem]{Lemma}

\newtheorem{remark}[theorem]{Remark}
\newtheorem{assumption}[theorem]{Assumption}

\journal{Transactions in Automatic and Control}

\newcommand{\VPDE}{V_{PDE}}
\newcommand{\Veff}{V_{eff}}
\newcommand{\VSDE}{V_{SDE}}
\newcommand{\Vin}{V}

\begin{document}

\begin{frontmatter}

%\title{Coupled $n \times m$ Hyperbolic SDE}
\title{Stabilization and Optimal Control of an Interconnected $n + m$ Hetero-directional Hyperbolic PDE - SDE System}
%\author{gvelho033 }
%\date{May 2024}

\affiliation[Université Paris-Saclay]{organization={Université Paris-Saclay, CentraleSupélec, CNRS, Laboratoire des signaux et systèmes},%Department and Organization 
            city={Gif-sur-Yvette},
            country={France}}
\affiliation[Université Paris-Saclay2]{organization={Université Paris-Saclay, CentraleSupélec, INRIA, Laboratoire des signaux et systèmes \& IPSA},%Department and Organization 
            city={Gif-sur-Yvette},
            country={France}}

\author[Université Paris-Saclay]{Gabriel Velho}
\author[Université Paris-Saclay]{Jean Auriol}
\author[Université Paris-Saclay2]{Islam Boussaada}
\author[Université Paris-Saclay]{Riccardo Bonalli}

\begin{abstract}
%% Text of abstract
In this paper, we design a controller for an interconnected system composed of a linear Stochastic Differential Equation (SDE) controlled through a linear hetero-directional hyperbolic Partial Differential Equation (PDE). Our objective is to steer the coupled system to a desired final state on average, while keeping the variance-in-time as small as possible, improving robustness to disturbances. By employing backstepping techniques, we decouple the original PDE, reformulating the system as an input delayed SDE with a stochastic drift. We first establish a controllability result, shading light on lower bounds for the variance. This shows that the system can never improve variance below strict structural limits. Under standard controllability conditions, we then design a controller that drives the mean of the states while keeping the variance bounded. Finally, we analyze the optimal control problem of variance minimization along the entire trajectory. Under additional controllability assumptions, we prove that the optimal control can achieve any variance level above the fundamental structural limit. 
%Finally, we address the optimal control of the delayed SDE, providing theoretical guarantees on the performance of the optimal control under additional assumptions on the system dynamics.
%In this paper, we design a controller for an interconnected system consisting of a linear Stochastic Differential Equation (SDE) actuated through a $n \times m$ linear hyperbolic Partial Differential Equation (PDE). Our approach aims to minimize the variance of the state of the SDE component. We leverage a backstepping technique to transform the original PDE into an uncoupled stochastic PDE. As such, we reformulate our initial problem as the control of a multi-input delayed SDE with a non-deterministic drift. We first establish a controllability result for these control systems, revealing the existence of a lower bound under which the covariance of the control system cannot be steered. Under standard controllability assumptions, we design a controller steering the mean of the states to zero while keeping its covariance bounded. As final step, we address the optimal control of the multi-input delayed SDE. Under some additional assumptions on the dynamics, we give theoretical guarantees on the efficiency of the optimal control. 
\end{abstract}

\end{frontmatter}

\section{Introduction}\label{INT}

Interest in interconnected systems combining Ordinary Differential Equations (ODEs) and Partial Differential Equations (PDEs) began when delays in ODEs were associated with transport PDEs, enabling a reinterpretation of the classical Finite Spectrum Assignment \cite{artstein_linear_1982}, as seen in~\cite{krstic_boundary_2008}. Interconnections involving hyperbolic PDEs and ODEs have since been successfully used to model various real-world applications, such as the propagation of torsional waves in drilling systems~\cite{aarsnes_torsional_2018}, control in deep-water construction vessels \cite{stensgaard_subsea_2010}, and heat exchanger systems. Lyapunov and backstepping methods have played a key role in developing stabilizing controllers for such systems \cite{di_meglio_stabilization_2018, auriol_robustification_2023, deutscher_backstepping_2017, wang_delay-compensated_2020}.

In realistic settings, these dynamical systems are frequently subject to disturbances, for instance from measurement noise, parameter uncertainties, or external disturbances \cite{bonalli_sequential_2022}. Such disturbances can substantially impact system dynamics. Therefore, effective mitigation of these uncertainties is critical for the reliability and safety of controlled systems \cite{chen_optimal_2016}. Stochastic Differential Equations (SDEs) provide a broad and accurate framework for modeling uncertain systems \cite{touzi_introduction_2013}. The use of stochastic control enables the design of robust stabilizing controllers that effectively mitigate the impact of random fluctuations. Stabilizing SDEs in expectation is a well-established technique, but ensuring bounded variance remains essential for reliability \cite{chen_optimal_2016}. %Robustness can also be enhanced by designing controllers to maintain low state variance \cite{chen_optimal_2016, liu_optimal_2023}.

Motivated by the previous remarks, introducing uncertainty modeling in interconnected PDE+ODE offers more robust settings for safe control of the system.
%In interconnected PDE+ODE systems, accurately modeling uncertainties that affect the ODE dynamics is often vital. 
Previous methods have aimed to stabilize PDE+ODE systems in the presence of bounded or model-based disturbances \cite{redaud_output_2024, deutscher_backstepping_2017}. SDEs extend this framework, offering a broader scope for disturbance modeling. Many approaches address SDEs with delays \cite{oksendal_maximum_2001,wang_lq_2013, velho_mean-covariance_2023}. While delays can often be interpreted as transport PDEs, more complex systems with multiple velocities and multiple distinct delays remain less explored. Considering such systems is crucial as they naturally arise in applications such as traffic flow and networked control systems, where different propagation rates lead to complex delay structures that impact stability and performance.
A first step toward such complexity was addressed in \cite{velho_stabilization_2024}, where a SDE was interconnected with a scalar hyperbolic PDE introducing a single delay. However, systems involving higher-dimensional coupled hyperbolic PDEs, which can model a wider range of dynamics \cite{auriol_robust_2022}, still present significant challenges.

This work focuses on the stabilization and variance minimization of the SDE in an interconnected PDE-SDE system, where the PDE component consists of $n+m$ coupled hyperbolic equations. This structure introduces $m$ distinct delays in the system, adding significant complexity to the control design.
%This work aims to advance control methods for interconnected PDE+SDE systems, by studying the case where the PDE component includes $n+m$ coupled hyperbolic equations. This configuration introduces $m$ distinct delays in the SDE, significantly increasing the complexity of the control problem.

In this paper, we build on the methodology previously introduced in \cite{velho_stabilization_2024} for the case of a scalar hyperbolic PDE inducing a single delay, extending it to a multidimensional setting. Our approach first extends techniques for deterministic PDE+ODE control to the stochastic setting, adapting them to PDE+SDE systems. We then introduce a novel approach to analyze and mitigate the effects of potential multiple delays within the SDE, providing implementable and computationally efficient control strategies for these advanced systems.

 Concretely our contribution is threefold:
 \begin{itemize}
     \item We study the controllability of the system through a novel approach, revealing how disturbances and delays jointly affect the variance of the SDE state and establishing a structural lower bound on the achievable variance.
     \item We propose a feedback law yielding stabilization of the PDE+SDE system, steering the expectation of the state to zero while keeping the variance bounded.
     \item We develop an optimal control strategy to minimize the SDE variance over an entire time horizon, improving robustness to disturbances \cite{bonalli_solving_2017,bonalli_optimal_2018,bonalli_continuity_2019}. Under standard controllability assumptions, this strategy can reduce the variance arbitrarily close to the limiting, structural lower bound.
 \end{itemize}
 Our approach leverages the backstepping methodology to transform the coupled PDE-SDE system into a input-delayed SDE with a stochastic drift. To achieve stabilization and variance minimization, we combine the Kalman decomposition technique with Artstein's transformation and Linear Quadratic (LQ) control techniques.
%\modifJA{Parler de l'approche (place?)}

 The paper is organized as follows. 
 In Section \ref{PRO}, we outline the problem formulation. Next, in Section \ref{TRA}, we utilize the backstepping transformation to reduce the problem to controlling an input delayed SDE with a random drift term, while also establishing the well-posedness of the system. In Section \ref{COV}, we study the controllability of the SDE and propose a stabilizing feedback law. Finally, in Section \ref{OPT}, we present an optimal control approach to minimize the variance of the SDE.

\section{Problem Formulation}\label{PRO}

\subsection{Notations}

Let us consider $(N,n,m) \in (\mathbb{N} \setminus \{ 0 \})^3$. We assume that state variables take values in $\mathbb{R}^N$, while control variables take values in $\mathbb{R}^m$. 
We assume that we are given a filtered probability space $(\Omega, \mathcal{F} \triangleq (\mathcal{F}_t)_{t \in [0,\infty)}, \mathbb{P})$, %For the sake of clarity in the exposition and without loss of generality, from now on, 
and that stochastic perturbations are due to a one-dimensional Wiener process $W_t$, which is adapted to the filtration $\mathcal{F}$. For clarity, we focus on a one-dimensional Wiener process, noting that all results extend trivially to the multidimensional case but with increased notational complexity.
%\textcolor{blue}{Comme dans \cite{yong_stochastic_1999}}. %We denote $\mathcal{F}_t$ the filtration induced by $W_t$.
Let $T > 0$ be some given time horizon.
The space the system state evolves in is denoted by  $\mathbb{X} \triangleq \mathbb{R}^N \times L^2([0,1] , \mathbb{R}^n) \times L^2([0,1] , \mathbb{R}^m)$.
We denote by $L_\mathcal{F}^2([0,T] , \mathbb{X})$ the set of square-integrable processes $P: [0,T]\times\Omega \to \mathbb{X}$ that are $\mathcal{F}$--progressively measurable, whereas the subset $C^2_\mathcal{F}([0,T] , \mathbb{X}) \subseteq L_\mathcal{F}^2([0,T] , \mathbb{X})$ contains processes whose sample paths are continuous.
The spaces of semi-definite and definite positive symmetric matrices in $\mathbb{R}^n$ are denoted by $\mathcal{S}^+_n$ and $\mathcal{S}^{++}_n$, respectively. 
If $X \in L_\mathcal{F}^2([0,T] , \mathbb{R}^n)$, we denote by $V_X(\cdot)$ its variance, that is
$
V_X(t) \triangleq \espE[ (X(t) - \espE[X(t)])^T (X(t) - \espE[X(t)]) ] \in \mathbb{R}.
 $ Finally, we recall that if $f_1$ and $f_2$ are two deterministic functions in $L^2([0,T], \mathbb{R})$, It\^o formula yields %then ~\cite{le_gall_brownian_2016} 
 \small
 \begin{align}
     \espE \left[  \left( \int_0^t f_1(s) dW_s  \right) \left( \int_0^t f_2(s) dW_s  \right) \right] = \int_0^t f_1(s) f_2(s) ds. \label{eq:quadratic_variation_of_stochastic_integral}
 \end{align}
\normalsize

Throughout this paper, we consider equalities between stochastic processes up to a modification. A stochastic process $X$ is a modification of a process $Y$ if $\forall t, \probP(X(t) = Y(t)) = 1$. All the expectations considered in this work are unaffected, as $ \espE[ \int^T_0 h(X(t)) \; \mathrm{d}t ] = \int^T_0 \espE[ h(X(t)) ] \; \mathrm{d}t = \int^T_0 \espE[ h(Y(t)) ] \; \mathrm{d}t = \espE[ \int^T_0 h(Y(t)) \; \mathrm{d}t ]$ for any continuous function $h$ such that $\espE[ h(X(t)) ]$ is finite. This convention is justified as we focus primarily on the expectations of continuous functions of the processes -- expectation, variance, etc. This allows us to simplify expressions without repeatedly noting their equivalence up to stochastic modifications, thus ensuring clarity and conciseness.

%\modifJA{Attention à l'utilisation de Mean/Expectations. Est-ce toujours équivalent?}

\subsection{Equations}

In this paper, we consider interconnected hyperbolic PDE+SDE systems of the form 
%\modifJA{Peut-être redire ici dans quel context ces systèmes peuvent apparaitre?}
\begin{equation}\label{eq:coupled_PDE_SDE}
\left\{
\begin{array}{l}
     dX(t) = ( A X(t) + B v(t,0) ) dt + \sigma(t) dW_t \\
     u_t(t,x) + \Lambda^+ u_x(t,x) = \Sigma^{++}(x) u(t,x) + \Sigma^{+-}(x) v(t,x) \\
     v_t(t,x) - \Lambda^- v_x(t,x) = \Sigma^{-+}(x) u(t,x) + \Sigma^{--}(x) v(t,x) \\
     u(t,0) = Q v(t,0) + M X(t) \\
     v(t,1) = R u(t,1) + \Vin(t) \\
     X(0) = X_0,~ u(0,x) = u_0(x),~ v(0,x) = v_0(x),
\end{array} 
\right.
\end{equation}
in the time-space domain $[0, T] \times [0,1]$, where 
$$
u = (u_1, ... , u_n)^T, \qquad v = (v_1, ... , v_m)^T.
$$
%\modifJA{L'entier $n$ n'a pas été défini jusqu'à présent.} 
The state of the system is $(X(t), u(t,x), v(t,x)) \in \mathbb{X} \triangleq \mathbb{R}^N \times L^2([0,1] , \mathbb{R}^n) \times L^2([0,1] , \mathbb{R}^m)$. The diagonal matrices $\Lambda^+$ and $\Lambda^-$ are defined as
$$
%\Lambda^+ = \left( \begin{array}{ccc} \lambda_1 & & (0) \\ & \ddots &  \\ (0) & & \lambda_n \end{array} \right), \qquad \Lambda^- = \left( \begin{array}{ccc} \mu_1 & & (0) \\ & \ddots &  \\ (0) & & \mu_m \end{array} \right)
\Lambda^+ = \textrm{diag} \left( \lambda_1, \dots , \lambda_n \right), \quad \Lambda^- = \textrm{diag} \left(   \mu_1, \dots , \mu_m  \right)
$$
%\modifJA{Utilise la convention diag pour les matrices diagonales? Ca permet de gagner un peu de place.}
with 
$$
0 < \lambda_1 < ... < \lambda_n, \qquad 0 < \mu_1 < ... < \mu_m.
$$
These transport velocities are assumed to be constant. 
%However, our results can be extended to space-dependent transport velocities at the cost of lengthy and technical computations. 
The in-domain coupling terms $\Sigma^{\cdot \cdot}$ are continuous functions. We assume, without loss of generality, that the diagonal entries of $\Sigma^{++}$ and $\Sigma^{--}$ are equal to zero \cite{coron_local_2013}. 
% \modifJA{La ref à mettre est commentée ci-dessous%@article{coron2013local,
%   title={Local exponential H\^{}2 stabilization of a 2$\backslash$times2 quasilinear hyperbolic system using backstepping},
%   author={Coron, Jean-Michel and Vazquez, Rafael and Krstic, Miroslav and Bastin, Georges},
%   journal={SIAM Journal on Control and Optimization},
%   volume={51},
%   number={3},
%   pages={2005--2035},
%   year={2013},
%   publisher={SIAM}
% }
% }
The boundary coupling terms $R$ and $Q$ are such that the open-loop PDE system (i.e. $\Vin = 0$), without the SDE and in the absence of the in-domain coupling terms $\Sigma^{\cdot\cdot}$, is exponentially stable. This condition is required to avoid an infinite number of unstable poles, which is necessary to guarantee the existence of robustness margins for the closed-loop system; see \cite[Theorem 4]{auriol_explicit_2019}.
%The boundary coupling terms $R$ and $Q$ verify $\Vert R Q \Vert < 1$ to avoid an infinite number of unstable poles, which is necessary to guarantee the existence of robustness margins for the closed-loop system \cite{auriol_delay-robust_2018} 
%\modifJA{Pas tout à fait la bonne condition en vectoriel. Il faudrait plutôt dire: "The boundary coupling terms $R$ and $Q$ are such that the open-loop PDE system (i.e. $V_{in}=0$) without the SDE and in the absence of the in-domain coupling terms $\Sigma^{\cdot\cdot}$ is exponentially stable. This condition is required to avoid an infinite number of unstable poles, which is necessary to guarantee the existence of robustness margins for the closed-loop system" et citer le Théorème 4 de mon papier 2019 avec Florent (explicit mapping)}
%\commentGV{Attention est ce que c'est la bonne condition pour les dimensions $> 1$?}.
%We also assume $Q \neq 0$ \commentGV{Idem, a changer, check transformations}. 
The matrices $A \in \mathbb{R}^{N\times N}, B \in \mathbb{R}^{N\times m}, M \in \mathbb{R}^{n\times N}$ are constant, and~$\sigma$ is a deterministic diffusion in $L^\infty([0,T], \mathbb{R}^N)$. Finally, $(X_0, u_0, v_0) \in \mathbb{X}$. % with $\mathbb{X} \triangleq \mathbb{R}^N \times L^2([0,1] , \mathbb{R}^n) \times L^2([0,1] , \mathbb{R}^m)$ the space our state evolves in. 
The control input $V$ is a signal in the space $L_{\mathcal{F}}^2([0,T] , \mathbb{R}^m)$.
%\modifJA{Attention est-ce $V$ on $V_{in}$. Il faut être cohérent dans les notations.}
The class of system~\eqref{eq:coupled_PDE_SDE} naturally appear, for example, when modeling a heat-exchanger connected to a temperature system subject to random perturbations (e.g., the temperature of a building disrupted by a random outside temperature or sunlight).
We make the following classic controllability assumption on the dynamic in order to effectively control the system.
\begin{assumption}\label{ass:controlability_of_the_SDE}
%The pair of matrices $(A, (B_1, \dots , B_m) )$ is controllable.    
The pair %of matrices
$(A, B )$ is controllable.   
\end{assumption}
The well-posedness of the system is proved in section \ref{TRA} after the backstepping change of variable, where we show that the state $(X(\cdot), u(\cdot , \cdot), v(\cdot, \cdot))$ is in $L^2_\mathcal{F}([0,T], \mathbb{X})$
%\modifJA{Préciser qu'on montrera que l'état restera dans $\mathbb{R}^N \times L^2([0,1] , \mathbb{R}^n) \times L^2([0,1] , \mathbb{R}^m)$. Peut-être donner un nom du genre $\chi=\mathbb{R}^N \times L^2([0,1] , \mathbb{R}^n) \times L^2([0,1] , \mathbb{R}^m)$ pour éviter de réécrire ce domain 36 fois.}

%\subsection{Assumptions}

%\begin{assumption}
%    We assume the pair of matrices $(A,B)$ to be controllable.
%\end{assumption}

%\begin{assumption}
%    We assume the pair of matrices $(A,(B_1, ... , B_m))$ to be controllable but not $(A,(B_1, ... , B_{m-1}))$.
%\end{assumption}

%\subsection{Well-Posedness}

%The well-posedness is studied after the backstepping feedback loop is presented. Indeed, the study of the target equation is easier than the original system.

\subsection{Objective}
The system \eqref{eq:coupled_PDE_SDE} naturally contains multiple feedback loops or couplings that can potentially introduce instabilities. Therefore, our objective is three-fold:
%\modifJA{Avant de décrire l'objectif, peut-être souligner les prbolèmes qui apparaissent en Open-loop? "The system naturally contains multiple feedback loops or couplings that can potentially introduce instabilities. Therefore our objective is three-fold"}
%Our objective is three-fold :Riccardo Bonalli: Through appropriate citations, say where for what this result may be useful
%\modifJA{Que veut dire straightforward? Mot un peu vide de sens. Veux-tu dire explicit?} 
\begin{enumerate}
    %\item We reformulate the control problem of interconnected PDE-SDE systems into a single-input delayed SDE
    \item Analyze controllability and variance steering: we establish a structural lower bound on the achievable variance of the SDE state, which depends on the decoupling strategy used to stabilize the PDE but remains valid for any subsequent control signal. Additionally, we show that open-loop controls can steer the variance arbitrarily close to this structural lower bound.
    \item Stabilize the PDE+SDE system with feedback control: Since open-loop controls are generally impractical due to the unobservable noise realizations $ dW_t $ \cite{touzi_introduction_2013}, we give a stabilizing feedback law. This stabilizes the coupled PDE+SDE system by driving the expectation of the state to zero with exponential convergence while keeping the variance bounded. Specifically, for  $x \in [0,1]$ and $t>\frac{1}{\mu_1}$, the state $(X,u,v)$ satisfies:
    \begin{equation}\label{eq:stabilization_mean_expectation}
    \begin{split}
    &  \Vert \espE[X(t)] \Vert + \Vert \espE[u(x,t)] \Vert + \Vert \espE[v(x,t)] \Vert \leq C e^{- \nu t} \times\\
   & \quad \times \bigl (\Vert \espE[X_0] \Vert + \Vert  \espE [u_0] \Vert + \Vert  \espE [v_0] \Vert \bigr ), \\
   &  \espE \left[ \Vert X(t)\Vert^2 \right]  + \espE \left[ \Vert u(x,t) \Vert^2 \right] + \espE \left[ \Vert v(x,t) \Vert^2 \right] \leq C
    \end{split}
    \end{equation}
    with $C>0$ depending on the system parameters, while $\nu$ is defined by the feedback gain.
    
    \item Design a control that minimizes the quadratic cost of the SDE state while penalizing control input at the boundary:
    \begin{align*}
    J(X, v(t,0) ) \triangleq & \ \espE \left[  \int_0^T X(t)^T Q(t) X(t)  \right. \\
    & \hspace{3em} + \left. \int_0^T v(t,0)^T R(t) v(t,0) \right]. 
    \end{align*}
    This approach minimizes the variance, enhancing robustness with respect to perturbations\cite{bonalli_sequential_2022}, and it is computationally feasible. Under standard controllability assumptions, this strategy achieves variances that are arbitrarily close to the theoretical minimum by reducing the penalty parameter $R$ toward zero.
\end{enumerate}

%\modifJA{Je pense qu'on peut mieux exprimer ces objectifs. On en discute ensemble.}

\section{Tracking problem within the PDE}\label{TRA}

%In system \eqref{eq:coupled_PDE_SDE}, the in-domain coupling terms $\Sigma^{\cdot\cdot}$ in the PDE can cause instabilities and disrupt the control of the SDE \cite{auriol_delay-robust_2018}. %\modifJA{Mettre une citation?}. 
%Therefore, we will first apply a backstepping transformation to move these coupling terms to the boundary at $x=1$, where the control input is applied. Then, we will design the input $V(t)$ as a sum of two control components: $V(t)=\VPDE(t)+\Veff(t)$. The term $\VPDE(t)$ will correspond to a classical backstepping control input,  used to cancel the (potentially) unstable coupling terms appearing at the controlled boundary. It would stabilize the PDE in the absence of the SDE. On the other hand, $\Veff(t)$ is the effective control that propagates through the PDE, allowing us to track a desired signal entering the SDE.

In system \eqref{eq:coupled_PDE_SDE}, the in-domain coupling terms $\Sigma^{\cdot\cdot}$ in the PDE may cause instabilities and disrupt the control of the SDE \cite{auriol_delay-robust_2018}. %\modifJA{Mettre une citation?}. 
Our idea to cope with this challenge is to apply backstepping in the first place to move these coupling terms to the boundary at $x=1$, where the control input is applied. Then, we design the input $V(t)$ as a sum of two control components: $V(t)=\VPDE(t)+\Veff(t)$. The term $\VPDE(t)$ corresponds to a classical backstepping control input,  used to cancel the (potentially) unstable coupling terms appearing at the controlled boundary. It would stabilize the PDE in the absence of the SDE. On the other hand, $\Veff(t)$ is the effective control that propagates through the PDE, allowing us to approximately track a desired signal entering the SDE with a delay. This enables the application of delay stochastic control tools to steer the SDE and minimize its variance.

%\modifJA{J'ai bougé la fin ailleurs car ça faisait trop de choses à assimiler.}

%\subsection{Transformation into a Target System}
%\subsubsection{The Backstepping Transformation}
%\subsubsection{Existence of a Solution of the Kernel Equations}

\subsection{The Backstepping Transformation}

Taking inspiration from \cite{auriol_delay-robust_2018}, we consider the following backstepping change of variables:

\begin{equation}\label{eq:backstepping_transform_alpha_beta_definition}
 \begin{split}
 \alpha(t,x) \triangleq u(t,x) & + \int_0^x K_{uu}(x,y) u(t,y) dy \\ 
 & + \int_0^x K_{uv}(x,y) v(t,y) dy + \gamma_\alpha(x) X(t) \\
 \beta(t,x) \triangleq v(t,x) & + \int_0^x K_{vu}(x,y) u(t,y) dy \\ 
 & + \int_0^x K_{vv}(x,y) v(t,y) dy + \gamma_\beta(x) X(t).
 \end{split}
\end{equation}
The kernels  $K_{uu}, K_{uv}, K_{vu}$, and $K_{vv}$ are piecewise continuous functions defined on the triangular domain $\mathcal{T} \triangleq \left\{ (x,y), x \in [0,1], y \in [0,x]  \right\}$, whereas $\gamma_\alpha, \gamma_\beta$ are in $C^1([0,1] ; \mathbb{R}^{n \times N}) \times C^1([0,1] ; \mathbb{R}^{m \times N})$. On their respective domains of definition, $K$ and $\gamma$ verify the following set of \textit{kernel equations}: 

%\commentGV{Changement d'ordre par rapport à CDC, on introduit d'abord le target système pour introduire les fonctions $\Psi$ et $\Omega$.}

\small

 \begin{equation}\label{eq:kernel_equations_alpha}
 \left\{
 \begin{array}{l}
      \Lambda^+ (K_{uu})_x(x,y) + (K_{uu})_y(x,y) \Lambda^+ \\ 
      \quad + K_{uu}(x,y) \Sigma^{++}(y) + K_{uv}(x,y) \Sigma^{-+}(y) = \Psi(x) K_{uu}(x,y) \\
      \Lambda^+ (K_{uv})_x(x,y) - (K_{uv})_y(x,y) \Lambda^- \\
      \quad + K_{uu}(x,y) \Sigma^{+-}(y) + K_{uv}(x,y) \Sigma^{--}(y) =  \Psi(x) K_{uv}(x,y) \\
      \Lambda^+  K_{uu}(x,x) - K_{uu}(x,x) \Lambda^+ + \Sigma^{++}(x) = 0 \\
      \Lambda^+  K_{uv}(x,x) + K_{uv}(x,x) \Lambda^- + \Sigma^{+-}(x) = 0 \\
      K_{uu}(x,0) \Lambda^+ Q - K_{uv}(x,0) \Lambda^- + \gamma_\alpha(x) B = \Psi(x)  \\
      \Lambda^+ \gamma_\alpha'(x) + \gamma_\alpha(x) A + K_{uu}(x,0) \Lambda^+ M  = 0 \\
      \gamma_\alpha(0) = -M,
 \end{array} 
 \right.
 \end{equation}
 \normalsize
 as well as
%\modifJA{S'il faut gagner de la place, on pourra utiliser la forme condensée proposée dans LCSS ou mon papier Automatica de 2023 (robustification of ...)} 
\small
\begin{equation}\label{eq:kernel_equations_beta}
 \left\{
 \begin{array}{l}
      - \Lambda^- (K_{vu})_x(x,y) + (K_{vu})_y(x,y) \Lambda^+ \\
      \quad + K_{vu}(x,y) \Sigma^{++}(y) + K_{vv}(x,y) \Sigma^{-+}(y) = \Omega(x) K_{vu}(x,y) \\
      - \Lambda^- (K_{vv})_x(x,y) - (K_{vv})_y(x,y) \Lambda^- \\
      \quad + K_{vu}(x,y) \Sigma^{+-}(y) + K_{vv}(x,y) \Sigma^{++}(y) = \Omega(x) K_{vv}(x,y) \\
      - \Lambda^-  K_{vu}(x,x) - K_{vu}(x,x) \Lambda^+ + \Sigma^{-+}(x) = 0 \\
      - \Lambda^-  K_{vv}(x,x) + K_{vv}(x,x) \Lambda^- + \Sigma^{--}(x) = \Omega(x) \\
      K_{vu}(x,0) \Lambda^+ Q - K_{vv}(x,0) \Lambda^- + \gamma_\beta(x) B = 0 \\
      - \Lambda^- \gamma_\beta'(x) + \gamma_\beta(x) A  + K_{vu}(x,0) \Lambda^+ M = 0 \\
      \gamma_\beta(0) = 0.
 \end{array} 
 \right.
 \end{equation}
  \normalsize
%\commentGV{retrouver ou interviennent $\Psi$ et $\Omega$, les équations ne sont pas bien posées telles qu'elles sont.}
where $\Psi$ and $\Omega$ are two strict upper triangular matrices. %of the form \modifJA{Pas nécessaire de metttre explicitement les matrices. Tu dis déjà qu'elles sont trig. sup.}
% \small
% $$
% \Psi(x) = \left( \begin{array}{ccc} 0 & & (\psi(x))_{i,j} \\ & \ddots &  \\ (0) & & 0 \end{array} \right),  \Omega(x) = \left( \begin{array}{ccc} 0 & & (\omega(x)))_{i,j} \\ & \ddots &  \\ (0) & & 0 \end{array} \right).
% $$
% \normalsize
Their coefficients verify
\begin{align*}
\psi_{i,j}(x) & =  ( \lambda_i - \lambda_j) K_{uu}(x,x)_{i,j} + \Sigma^{++}_{i,j} , \\
\omega_{i,j}(x) & = ( \mu_j - \mu_i) K_{vv}(x,x)_{i,j} + \Sigma^{--}_{i,j}.
\end{align*}
The set of equations \eqref{eq:kernel_equations_alpha}-\eqref{eq:kernel_equations_beta} is well posed and admits a unique solution \cite{auriol_minimum_2016, di_meglio_stabilization_2018}. % \modifJA{Il faudrait faire référence au numéro d'équations. Par ailleurs in me semble qu'il y a un problème de conditions de bords pour $K_{uu}(x,0)$. Ca pourrait avoir des conséquence importantes sur le système cible (mais mineures sur le reste du papier). Il faudrait je pense remplacer $\Psi(x)\alpha(t,x)$ par $\Psi(x)\beta(t,0)$ En ref, il faut aussi citer le papier de Di Meglio/Hu/Bribiesca-Argomedo/Krstic (gestion du term ODE).}
By setting our stabilizing backstepping controller to
\begin{equation}\label{eq:definition_expression_backstep_controller}
\begin{split}
& \VPDE(t) = - R u(t,1) - \gamma_\beta(1) X(t) \\
    & \ - \int_0^1 K_{vu}(1,y) u(t,y) dy - \int_0^1 K_{vv}(1,y) v(t,y) dy ,
\end{split}
\end{equation}
and by differentiating with respect to time and space (in the sense of distributions), integrating by parts one shows that the states $(\alpha,\beta,X)$ are solutions to the following \textit{target system}:
\begin{equation}\label{eq:first_target_system_SPDE}
\left\{
\begin{array}{l}
     dX(t) = ( A X(t) + B \beta(t,0) ) dt + \sigma(t) dW_t \\
     d \alpha(t,x) + \Lambda^+ \alpha_x(t,x) = \Psi(x)\beta(t,0) + \gamma_\alpha(x) \sigma(t) dW_t \\
     d \beta(t,x) - \Lambda^- \beta_x(t,x) = \Omega(x) \beta(t,x) + \gamma_\beta(x) \sigma(t) dW_t \\
     \alpha(t,0) = Q \beta(t,0), \ \beta(t,1) = \Veff(t) .
\end{array} 
\right.
\end{equation}

%\modifJA{Mettre ce qui suit en remarque:
In the absence of the SDE, this target system would be exponentially stable. However, due to the cancellation of the term $Ru(t,1)$, the feedback operator induced by $\VPDE$ is not strictly proper. This could lead to (delay)-robustness issues, as highlighted in~\cite{auriol_delay-robust_2018}. To address these concerns, different approaches have been proposed as partial cancellation of the reflection~\cite{auriol_delay-robust_2018} or filtering methods~\cite{auriol_robustification_2023}. However, these methods do not directly adapt to the current stochastic configuration. Yet, we should be capable of showing robust mean square stability, but quantifying the impact of these techniques on variance proves to be challenging. This robustness analysis is beyond the scope of this paper and is left for future research.
%}

%\revAns{Attention à introduire Vin, Veff et VBS comme dans CDC dans le problem formulation.}

%\subsection{Signal Tracking}

%Although the coupling between $u$ and $v$ is handled, there remain cascading interactions between the states $v^{(i)}(t,x)$ and $v^{(j)}(t,x)$ for $i>j$. 
%Additionally, we need to account for the interaction from the SDE back into the PDE. 
While system \eqref{eq:first_target_system_SPDE} is simpler to handle, expressing $\beta(t,0)$ in terms of $\Veff$ remains complex, complicating the task of tracking a desired control through the PDE when plugged into the SDE. In what follows, we show how we can further simplify system   \eqref{eq:first_target_system_SPDE} into a input-delayed SDE, with an arbitrary input $\VSDE \in L^2_\mathcal{F}([0,T] ; \mathbb{R}^m)$, by choosing an appropriate controller $\Veff$.
%We therefore show that, for any desired signal $V_{SDE}(t)$ in $L^2_\mathcal{F}([0,T] ; \mathbb{R}^m)$ that we want to track into the SDE,  $\Veff(t)$ can be chosen to ensure the boundary condition $v^{(i)}(t,0) = V_{SDE}^{(i)}(t-\frac{1}{\mu_i}) + \int_{t - \frac{1}{\mu_i}}^t G_i(t-s ) \sigma(s) dW_s $ is met. The delay reflects the finite control velocity, and the stochastic integral arises from the noise that cannot be canceled.
%\modifJA{$\VSDE$ n'est pas proprement défini. On ne comprend pas à quoi ça correspond.}

\subsection{Solving the SPDE Through the Characteristic Method}
%\modifJA{Balise à garder pour me rappeler de relire cette section une fois modifiée.}

To achieve our desired control behavior at the boundary of the SDE, we start by considering an input $\Veff$ applied to the PDE. Due to the coupling terms within the PDE in system \eqref{eq:first_target_system_SPDE}, this input does not directly translate to a simple expression for $\beta(t,0)$. However, we find that, by choosing $\Veff$ in a specific form relative to an arbitrary process $S(t) \in L^2_\mathcal{F}([0,T] ; \mathbb{R}^m)$, we obtain a clean boundary condition at the other end of the PDE: $\beta(t,0) = S(t-\frac{1}{\mu_1}) + \int_{t - \frac{1}{\mu_1}}^t G(t-s ) \sigma(s) dW_s $.
This result implies that any desired control $\VSDE$ can be tracked at $\beta(t,0)$ up to a noise $\int_{t - \frac{1}{\mu_1}}^t G(t-s ) \sigma(s) dW_s $ when choosing $S(t) = \VSDE(t)$. The noise is independent of the control signal and can therefore not be compensated for.
%We cannot track the process any better, as the delay reflects the finite control velocity, and the stochastic integral arises from the noise that cannot be canceled.
%We can now solve through the characteristic method and the Ito formula the SPDE \eqref{eq:first_target_system_SPDE}. As mentioned previously, this allows us to compute an effective control $\Veff$ that tracks a desired control $\VSDE$ into the SDE.
%For clarity and conciseness, we only do this for $\beta$. %, although a very similar formula can be obtained for $\alpha$.
%\modifJA{Paragraphe peu clair. A reformuler. Pourquoi on a besoin de celle de $\alpha$?}
%\commentGV{A REVOIR: l'expression du contrôle $\Veff$ ne peut pas dépendre de $x$! Il faut en fait prendre l'expression en $x=0$. Le problème c'est que l'expression de $\beta(t,x)$ va changer sur le reste du domaine et on aura une induction qui sera plus complexe que ça. La méthode reste correcte, mais l'expression de $\Veff$ est beaucoup plus compliqué.}
%\commentGV{A REVOIR: Le contrôleur proposé n'est pas causal, il faut revoir la preuve et affiner l'expression de $\beta$.      }
\begin{lemma}\label{lem:track_VSDE_with_Veff_and_formula_beta}
    Let $\VSDE$ be in $ L^2_\mathcal{F}([0,T] ; \mathbb{R}^m)$ . %\modifJA{Qu'est-ce que ça veut dire ? Il faut introduire $\VSDE$ plus clairement et proprement.} 
    The input $\Veff \in  L^2_\mathcal{F}([0,T] ; \mathbb{R}^m)$, recursively defined as 
    %such that \modifJA{Domaine de $t$ et $x$?}
    \small
\begin{equation}\label{eq:def_of_Veff_in_terms_of_VSDE}
\begin{split}
% &  \Veff^{(m)}(t)  = \VSDE^{(m)}(t)   \\
% & \Veff^{(i)}(t)  = \VSDE^{(i)} \left(t \right) \\
% & - \sum_{j=i+1}^m  \int_{0}^{\frac{1-x}{\mu_i}} \omega_{i,j}(x + \mu_i s ) \times \\
% & \hspace{6em} \VSDE^{(j)} \left(t + \frac{1-x}{\mu_i} + s - \frac{1 - x - \mu_i s}{\mu_j} \right)  ds 
%\Veff^{(i)}(t- \frac{1}{\mu_i} ) & = \VSDE^{(i)}(t- \frac{1}{\mu_i} ) \\
%& -  \sum_{j=i+1}^m \int_0^{\frac{1}{\mu_i}} F_{ij}(0,s) \Veff^{(j)}(t-s) ds
\Veff^{(i)}(t) & = \VSDE^{(i)}\left(t + \frac{1}{\mu_i} - \frac{1}{\mu_1} \right) \\
& -  \sum_{j=i+1}^m \int_{\frac{1}{\mu_j}}^{\frac{1}{\mu_i}}  F_{ij}(0,s) \Veff^{(j)}\left(t + \frac{1}{\mu_i} - s\right) ds
\end{split}
\end{equation}
\normalsize
with $F_{ij}$ a sequence of functions defined recursively in the Appendix in \eqref{eq:definition_2_F_multi_Veff} and \eqref{eq:definition_1_F_multi_Veff},
yields the following expression for $\beta(t,0)$:
\begin{equation}\label{eq:induction_hypo_lemma_tracking_delayed_SDE_beta}
\begin{split}
    \beta^{(i)}(t,0) & = \VSDE^{(i)} \left(t - \frac{1}{\mu_1} \right) \\ 
    \quad & + \int_{t - \frac{1}{\mu_1}}^t G_i(0, t-s ) \sigma(s) dW_s
\end{split}
\end{equation}
    where $G_i$ is a deterministic $C^1$ function defined recursively:
\small
\begin{equation}\label{eq:def_of_G_state_feedback_noise}
\begin{split}
    & G_m(x,u) \triangleq \gamma_\beta^{(m)}( x + \mu_m u )   \\
    & G_i(x,u) \triangleq \gamma_\beta^{(i)}(x + \mu_i u) \qquad \qquad i < m \\ 
    & + \sum_{j=i+1}^m \int_0^{\overline{s}_{i,j}(x,u)} \omega_{i,j}\bigl (x + \mu_i (u-s) \bigr) G_j \bigl (x + \mu_i (u-s), s \bigr )   ds ,
\end{split}
\end{equation}
\normalsize
where $\overline{s}_{i,j}(x,u) \triangleq \min \left(  \frac{1-x-\mu_i u}{\mu_j - \mu_i} , u\right)$. 
\end{lemma}

% can we do better with \beta^{(i)}(t,x) = \VSDE \left(t - \frac{1-x}{\mu_i} \right) + \int_{t - \frac{1-x}{\mu_i}}^t G_i(x + \mu_i(t-s)) \sigma(s) dW_s 
\begin{proof}
Inspired by the proof of \cite{hu_control_2016}, since $\Omega$ is strictly upper triangular, we proceed by induction. We first obtain an explicit expression for $\beta_m(t,x)$. Then, for any $i \in [1,m-1]$, we give a formula that expresses $\beta_i(t,x)$ in terms of the other $\beta_j$, $j \in [i+1,m]$. This allows us to recursively compute $\beta(t,x)$, and then to compute a controller $\Veff$ that tracks the desired control $\VSDE$. % The obtained controller $\Veff$ is given by
The details of the proof can be found in Appendix \ref{APP_LEM_TRACK_NEW}.
\end{proof}
%\modifJA{Je donnerais plutôt le résultat dans l'autre sens. On envoie ce $V_{eff}$ et ça donne ça. En particulier $\beta_i(t,0)=v_i(t,0)$ vérifie... On ne peut pas faire mieux pour des raisons de causalité.}
%In expression \eqref{eq:induction_hypo_lemma_tracking_delayed_SDE_beta}, we do not track the desired controller $\VSDE$ perfectly : it is subject to a delay as well as a noise in the form of a stochastic integral. The delay is due to the finite velocity of propagation of the signal through the transport PDE. If we are to keep an adapted controller (that is one that does not predict future noise), at component $i$ we cannot track at the point $(t,x)$ the signal $\VSDE^{(i)}$ with a delay lower than $\frac{1-x}{\mu_i}$ (as we cannot use future values of $\VSDE^{(i)}$ because then $\Veff$ would not be adapted). We also see in that regard that we cannot compensate the noise any better, as it is given by a stochastic integral between $t - \frac{1-x}{\mu_i}$ and $t$ that is therefore independent of the signal $\VSDE^{(i)}(t - \frac{1-x}{\mu_i})$. 
In expression \eqref{eq:induction_hypo_lemma_tracking_delayed_SDE_beta}, we cannot perfectly track the desired control $\VSDE$; it is affected by both a delay and a stochastic noise term. 
The delay arises from the finite propagation speed of the signal through the transport PDE, while the noise results from the SDE state re-entering the PDE at the boundary $x=0$.
%The delay stems from the finite propagation velocity of the signal through the transport PDE. 
%The noise is in turn caused by the state of the SDE re-entering the PDE at the boundary $x=0$.
The noise cannot be fully compensated, as it depends on a stochastic integral on $\left[t - \frac{1}{\mu_1}, t \right]$, which is independent from $\VSDE\left(t - \frac{1}{\mu_1} \right)$.
%For the component $i$, the control $\VSDE^{(i)}$ cannot be tracked at point $x$ with less delay than $\frac{1-x}{\mu_i}$ without predicting future noise, which would make $\Veff$ non-adapted to the filtration $\mathcal{F}$.
%The control $\VSDE$ cannot be tracked at the boundary $x=0$ without delay, and without predicting future noise, which would make $\Veff$ non-adapted to the filtration $\mathcal{F}$. 

%\commentGV{Write a remark on the fact that the control $\Veff$ cannot better track $\VSDE$ because $\Veff(t-h_i)$ and the stochastic integral $\int_{t-h_i}^t (...) \sigma(s) dW_s $ are independent.}]

\begin{remark}
%The speed of the transport PDE at component $i$ is $\mu_i$, which, in theory, allows us to optimally track the signal $\VSDE^{(i)}\left(t- \frac{1}{\mu_i} \right) $ at $\beta_i(t,0)$, which would be minimal. However, in the general case, this approach introduces an additional term in the other components $j<i$, of the form
Setting the boundary condition $\beta_i(t,0) = \VSDE^{(i)}\left(t- \frac{1}{\mu_i} \right) + \int_{t - \frac{1-x}{\mu_i}}^t G_i(0, t-s ) \sigma(s) dW_s$ would enhance performances of the controller, as it would reduce the input delay in the SDE in some components. However, it introduces an additional term in the other components $j<i$, of the form
$$
\beta_j(t,0) = \VSDE^{(j)}\left(t - \frac{1}{\mu_1}\right) + \int_{\frac{1}{\mu_i}}^{\frac{1}{\mu_1}} N_{ij}(s) \VSDE^{(i)}(t-s) ds + \dots
$$
This additional term cannot be compensated without predicting future values of the signal $\VSDE$. Consequently, such tracking would lead to an equivalent input-delayed SDE with distributed input delays. To the best of our knowledge, there are no established results on the control of SDEs with distributed input delays. Therefore, we track the signal $\VSDE^{(i)}\left(t - \frac{1}{\mu_1}\right)$ uniformly across all components. This approach allows us to develop an efficient stabilizing controller for the interconnected system. % while minimizing the variance. 
The study of input-delayed SDEs with distributed input delays is left for future work.
%In the general case, tracking a signal $\VSDE^{(i)}$ on the boundary with a delay $\delta_i \in [\frac{1}{\mu_i} ,\frac{1}{\mu_1}] $ implies an additional term in the following components $j<i$ of the boundary of the form
%$$\beta_j(t,0) = \VSDE^{(j)}(t - \delta_j) + \int_{\delta_i}^{\delta_j} N_{ij}^\delta(s) \VSDE^{(i)}(t-s) ds + (...)$$
%which cannot be compensated without predicting the future signal if $\delta_i<\delta_j$. The input delay SDE associated would contain multiple distributed input-delays. To the extent of our knowledge, such SDEs have not been considered. Minimizing the variance of such systems seems out of reach, so our strategy is therefore (in the general case) to track the signal $\VSDE$ with only one delay $\mu_1$ at the boundary $\beta(t,0)$. We may lose on some performance, but the problem remains tractable. In some cases however (when there is no coupling), we can track in minimal time the signal without any additional distributed input delayed terms.
\end{remark}

\subsection{Well-Posedness and Regularity of the Solution}

%\modifJA{Peut-être donner la tête de celle de $\alpha$?}
Using the explicit expressions for $\alpha$ and $\beta$ given in \eqref{eq:induction_hypo_lemma_tracking_delayed_SDE_beta}, we can now prove the well-posedness of the target system \eqref{eq:first_target_system_SPDE} directly. 
\begin{lemma}\label{lem:well_posed_SPDE_L2}
    If $\VSDE$ is in $C^2_\mathcal{F}([0,T]; \mathbb{R}^m)$, and $\Veff$ is defined accordingly by \eqref{eq:def_of_Veff_in_terms_of_VSDE}, % \modifJA{dire que $\Veff$ est defini par eq (10)}
    then the closed-loop target system~\eqref{eq:first_target_system_SPDE} is well-posed, that is there exist unique processes %\modifJA{Dire ce que ça signifie} 
    $t \mapsto \alpha(t,\cdot)$ and $t \mapsto \beta(t,\cdot)$ in $C^2_\mathcal{F}([0,T] ; L^2(0,1))$ and a unique process $X \in C^2_\mathcal{F}([0,T] ; \mathbb{R}^n)$ that satisfy \eqref{eq:first_target_system_SPDE}.
\end{lemma}
\begin{proof}
    The proof is identical to the one proposed in~\cite[Lemma 1]{velho_stabilization_2024} for $n=m=1$.
\end{proof}
Since the backstepping transformation is a Volterra transformation, it is boundedly invertible \cite{yoshida_lectures_1960}. This implies the well-posedness of the original system~\eqref{eq:coupled_PDE_SDE}.

\subsection{A delayed SDE}

Lemma \ref{lem:track_VSDE_with_Veff_and_formula_beta} states that for any adapted control $\VSDE$, there exists a controller $\Veff$ such that
%$$\beta(t,0) = \left( \begin{array}{c} \VSDE^{(1)} \left(t - \frac{1}{\mu_1} \right) \\ \vdots \\ \VSDE^{(m)} \left(t - \frac{1}{\mu_m} \right) \end{array} \right) + \int_{t-\frac{1}{\mu_1}}^t G(t-s) \sigma(s) dW_s $$
$$
\beta(t,0) = \VSDE \left(t - \frac{1}{\mu_1} \right) + \int_{t-\frac{1}{\mu_1}}^t G(t-s) \sigma(s) dW_s
$$
%Note that, to streamline notation, we express the noise integral uniformly as $\int_{t-\frac{1}{\mu_1}}^t G(t-s) \sigma(s) dW_s$, where we recall that $\frac{1}{\mu_1}$ represents the largest delay across all components. It is implicitly implied that the components $G_i$ are zero on the intervals  $[\frac{1}{\mu_i}, \frac{1}{\mu_1}]$. This choice avoids detailing each component’s individual delay while capturing the essential structure of the system's noise. 
%Combining these results, we demonstrate that for any desired adapted signal $\VSDE$, we can compute a controller $V(t) = \VPDE + \Veff$, where $\VPDE$ is given by \eqref{eq:expression_backstep_controller} and $\Veff$ is given by \eqref{eq:def_Veff_in_terms_of_VSDE} such that the process 
Consequently, %in the general setting
the process $X(t)$ defined in equation~\eqref{eq:coupled_PDE_SDE} is the solution of the following input-delayed stochastic differential equation with random coefficients 
\begin{equation}\label{eq:equivalent_delayed_SDE}
\left\{
\begin{array}{l}
     dX(t) = \left ( A X(t) + B U(t-h) + r(t) \right ) dt + \sigma(t) dW_t  \\
     X(0) = X_0, \quad U(s) = \beta(0, 1 + \mu s) \  \forall s \in [-h,0)
\end{array} 
\right.
\end{equation}
where $h \triangleq \frac{1}{\mu_{1}}$, $U \triangleq \VSDE$, and 
$$
r(t) \triangleq B \int_{t-h}^t G(t-s) \sigma(s) dW_s.
$$

\begin{remark}
In certain classes of $n+m$ hyperbolic transport PDEs, the target system %
% \begin{equation}\label{eq:first_target_system_SPDE_no_coupling}
% \left\{
% \begin{array}{l}
%      dX(t) = ( A X(t) + B \beta(t,0) ) dt + \sigma(t) dW_t \\
%      d \alpha(t,x) + \Lambda^+ \alpha_x(t,x) = \Psi(x)\beta(t,0) + \gamma_\alpha(x) \sigma(t) dW_t \\
%      d \beta(t,x) - \Lambda^- \beta_x(t,x) = \gamma_\beta(x) \sigma(t) dW_t \\
%      \alpha(t,0) = Q \beta(t,0), \ \beta(t,1) = \Veff(t)
% \end{array} 
% \right.
% \end{equation}
 \eqref{eq:first_target_system_SPDE} is such that the matrix $\Omega(x)$ is null for all $x$. This occurs, for instance, when there are no couplings between the transport PDEs. In this system, there are no interactions between the components of $\beta(t,x)$:
$$
\beta^{(i)}(t,0) = \Veff^{(i)} \left(t - \frac{1}{\mu_i} \right) + \int_{t-\frac{1}{\mu_i}}^t G_i(t-s) \sigma(s) dW_s.
$$
In such setting, we are essentially tracking a control signal at the boundary $\beta(t,0)$ in minimal time, which allows a more precise control of the interconnected PDE-SDE system. The resulting input-delayed SDE is of the form
\small
\begin{equation}\label{eq:equivalent_multi_delayed_SDE}
\left\{
\begin{array}{l}
     dX(t) = \left ( A X(t) + \sum_{i=1}^m B_i U_i(t-h_i) + r(t) \right ) dt + \sigma(t) dW_t  \\
     X(0) = X_0, \quad U(s) = \beta(0, 1 + \mu s) \  \forall s \in [-h,0) ,
\end{array} 
\right.
\end{equation}
\normalsize
where $h_i \triangleq \frac{1}{\mu_{m-i}}$, $U_i \triangleq \Veff^{(m-i)}$, and %\modifJA{Attention $V_i$ fait trop penser à la $i$ème composante de $V$ peut être l'appeler $\bar V_i$? A moins que ce soit $U$?}
$$
r(t) \triangleq B \int_{t-h_m}^t G(t-s) \sigma(s) dW_s.
$$
Note that the components of $\Veff$ have been reversed to respect the convention $h_1 < ... < h_m$. % \modifJA{Pas très clair}
Such minimal time tracking of the control signal enables to further decrease the variance of the SDE, yielding a more efficient control as we detail in the next section.
%\commentGV{Faire un commentaire sur les changements de notations fait ici. Il est important d'inverser les composantes du contrôle pour faciliter la compréhension de la forme de Kalman qui suit.}
\end{remark}
%Therefore, we now only focus on controlling the delayed SDE~\eqref{eq:equivalent_delayed_SDE} and minimizing the variance of the state. We follow a methodology similar to the one outlined in \cite{velho_mean-covariance_2023}. However, the inclusion of the extra delays requires to modify our approach to minimize the variance, as the Artstein predictor is more complex.
From now on, we focus on steering the multi-input delayed SDE~\eqref{eq:equivalent_multi_delayed_SDE} while minimizing the state variance. We focus on this more general multi-input delayed system given that the results developed in the next section can be applied to systems with only one delay as a particular case. We follow a methodology similar to the one outlined in \cite{velho_mean-covariance_2023}. However, compared to~\cite{velho_mean-covariance_2023}, the additional delays $h_i$ that appear in equation~\eqref{eq:equivalent_delayed_SDE} require several adjustments in our variance-minimization approach, given that the computation of Artstein's predictor becomes more complex.

%\section{Stabilization of the PDE+SDE}\label{STA}
%\input{04_5_Stabilization}

\section{Covariance Steering of the SDE}\label{COV}

Our goal now is to minimize the system's covariance as effectively as possible. Following the approach outlined in \cite{velho_mean-covariance_2023}, we first establish a fundamental lower bound on the system's covariance. %After this, we demonstrate that open-loop controllers can achieve a covariance value that is arbitrarily close to this lower bound within a finite amount of time.
%Note that by open-loop control, we refer to an $\mathcal{F}_t$ adapted process that is not necessarily a direct function of the state, and that possibly requires the realizations of $dW_t$ for computation.

%\modifJA{What do you mean by open-loop controllers?}

%\subsection{Simplifying the system combining the Kalman canonical form and the Artstein predictor}
%\modifJA{Titre de sous-section trop long.}
\subsection{Minimal variance computation}

The minimal bound for covariance can be computed easily when there is a unique delay, primarily using Artstein's predictor \cite{artstein_linear_1982, velho_stabilization_2024}. % \modifJA{Ref pour cette affirmation et pour le prédicteur Artstein}. 
However, under multiple delays, computing this bound becomes more complex as it is influenced by the range of action of each component of the controller on the system state. To address this, we employ the Kalman decomposition \cite{dahleh_mit_2011} to change variables in a way that clearly separates the subspaces controlled by each input component. We then use Artstein's predictors on each subspace, taking into account the corresponding delay. % to obtain a bound for the system's covariance.

\begin{theorem}[Kalman canonical form extended]\label{thm:Kalman_canonical_extended}
For any control $U \in L^2_\mathcal{F}([0,T]; \mathbb{R}^m)$, let $X$ be the solution of \eqref{eq:equivalent_delayed_SDE},  %\modifJA{Dire que c'est pour toute loi de commande $U$? Faut-il des hypothèse supplémentaires?}, such that $\bigl (A, (B_1, ... , B_m ) \bigr)$ controllable \modifJA{Cette hypothèse doit être donnée et commentée dès la section 2!}. 
Then, under assumption \ref{ass:controlability_of_the_SDE}, there exists an invertible matrix $M$ and a change of variable $Z = MX$ such that $Z$ follows the dynamic %\modifJA{$M$ est une matrices? Si oui, plutôt dire, il existed une matrice reversible $M$ telle que ... et ...}
$$
dZ = \left( \overline{A} Z + \overline{B}\left( \begin{array}{c} U_1(t-h_1) \\ \vdots \\ U_m(t-h_m) \end{array} \right) + \overline{r}(t) \right) dt + \overline{\sigma} dW_t
$$
where $\overline{A}$ and $\overline{B}$ are triangular-by-block matrices given by $\overline{A} \triangleq M A M^{-1}$ and $\overline{B} \triangleq M B$, and where $\overline{\sigma} \triangleq M \sigma$, $\overline{r} \triangleq \int_{t-h}^t \overline{G}(t-s) \overline{\sigma}(s) dW_s $ and $\overline{G}(u) \triangleq M B G(u) M^{-1}$. 
%\modifJA{La fin n'a pas à faire partie du théorème. C'es t plus une réécriture pour la convenance.} \commentGV{Pas vraiment vu qu'il faut préciser que $(\overline{A}_{ii} , \overline{B}_{ii})$ controllable.} 
We can rewrite the system %(\commentGV{with a slight abuse of notation}) 
as
\begin{equation}\label{eq:multi_input_delayed_SDE_kalman_form}
\begin{split}
    & \left( \begin{array}{c} dZ_1(t) \\ \vdots \\ dZ_m(t) \end{array} \right) = \left( \begin{array}{ccc} \overline{A}_{11} & &  \overline{A}_{ij} \\  & \ddots & \\ (0) & & \overline{A}_{mm} \end{array} \right)  \left( \begin{array}{c} Z_1(t) \\ \vdots \\ Z_m(t) \end{array} \right) dt \\ 
    & \hspace{3em} + \left( \begin{array}{ccc} \overline{B}_{11} & & \overline{B}_{ij} \\  & \ddots & \\ (0) & & \overline{B}_{mm} \end{array} \right) \left( \begin{array}{c} U_1(t-h_1) \\ \vdots \\ U_m(t-h_m) \end{array} \right) dt \\
    & \hspace{3em} +  \left( \begin{array}{c} \overline{r}_1(t) \\ \vdots \\  \overline{r}_m(t) \end{array} \right) dt + \left( \begin{array}{c} \overline{\sigma}_1(t) \\ \vdots \\  \overline{\sigma}_m(t) \end{array} \right) dW_t
\end{split}
\end{equation}
where $(\overline{A}_{ii} , \overline{B}_{ii})$ are controllable for all $i \in [1,m]$.
\end{theorem}
\begin{proof}
    The proof is provided in Appendix \ref{APP_THM_KALMANN}. %\modifJA{Sois plus précis. Donne la sous-section. }.
\end{proof}
\begin{remark}
    Note that, in this new triangular-by-block system, there could be fewer than $m$ rows. %\modifJA{Compared to what?} 
    %if not all the matrices are needed for controllability. 
    For instance, if $(A, (B_1, B_2) ) $ is controllable, a possible transform %\modifJA{a possible transform may lead to?}
    may lead to 
\begin{align*}
& dZ = \left( \begin{array}{cc} \overline{A}_{11} & \overline{A}_{12} \\ 0 & \overline{A}_{22} \end{array} \right) Z ~ dt \\
&  + \left( \begin{array}{cc} \overline{B}_{11} & \overline{B}_{12} \\ 0 & \overline{B}_{22} \end{array} \right) \left( \begin{array}{c} U_1(t-h_1) \\ U_2(t-h_2) \end{array} \right) dt   + \overline{r}(t) dt + \overline{\sigma}(t) dW_t.
\end{align*}
%\modifJA{Faire tenir dans la colonne}
In this case, setting the controls $U_i$ to zero for $i \geq 3$ would lead to the case where $m=2$ %and both matrices $B_1$ and $B_2$ are needed for the controllability \modifJA{Pas très clair}. 
\end{remark}
In what follows, we denote by $N_1, ..., N_m$ the respective dimensions of $Z_1, ..., Z_m $. We consider, without loss of generality, that all $m$ matrices are needed for controllability, meaning that $N_i \geq 1$ for all $i$.

In order to compute the structural minimal covariance bound, we now use Artstein's predictor~\cite{artstein_linear_1982} associated with each substate $Z_i$ and its corresponding delay $h_i$. For all $1 \leq i \leq m$ the dynamic of $Z_i$ is given by
\begin{equation}\label{eq:delayed_SDE_substate_Z}
\begin{split}
dZ_i(t) & = \left( \overline{A}_{ii} Z_i(t) + \overline{B}_{ii} U_i(t-h_i) \right)dt \\
& + \left( \sum_{j > i} \overline{A}_{ij} Z_j(t) + \sum_{j > i} \overline{B}_{ij} U_j(t-h_j) \right)dt \\
& + \overline{r}_i(t) dt + \overline{\sigma}_i(t) dW_t.
\end{split}
\end{equation}
The associated predictor $Y_i$ is given by
\begin{equation}\label{eq:Def_Artstein_predictor_substate_Zi}
\begin{split}
Y_i(t) = Z_i(t) + \int_{t-h_i}^t e^{ \overline{A}_{ii} (t-s - h_i)} \overline{B}_{ii} U_i(s) ds ,
\end{split}
\end{equation}
which satisfies the following SDE
\begin{equation}\label{eq:SDE_Artstein_predictor_substate_Zi}
\begin{split}
dY_i(t) & = \left( \overline{A}_{ii} Y_i(t) + \Tilde{B}_{ii} U_i(t) \right)dt \\
& + \left( \sum_{j > i} \overline{A}_{ij} Z_j(t) + \sum_{j > i} \overline{B}_{ij} U_j(t-h_j) \right)dt \\
& + \overline{r}_i(t) dt + \overline{\sigma}_i(t) dW_t,
\end{split}
\end{equation}
where $\Tilde{B}_{ii} \triangleq e^{ - \overline{A}_{ii} h_i} \overline{B}_{ii} $. Let us now introduce the process $\Tilde{Y}_i$ defined by %which mean and covariance are finite time controllable [\commentGV{citation}]:
\begin{equation}\label{eq:cropped_SDE_Artstein_predictor_tilde_Yi}
\begin{split}
d\Tilde{Y}_i(t) & = \left( \overline{A}_{ii} \Tilde{Y}_i(t) + \Tilde{B}_{ii} U_i(t) \right)dt + \overline{\sigma}_i(t) dW_t.
\end{split}
\end{equation}
Classical results in stochastic control \cite[Theorem 13]{mahmudov_controllability_2001} ensure that the mean and covariance of this process are controllable over a finite time horizon. This process serves as benchmark for determining which covariances are reachable and which are not. 
Additionally, we can establish a clear connection between the original substate $Z_i$ and $\Tilde{Y}_i$.

\begin{lemma}\label{lem:link_substate_cropped_predictor_tilde_Yi}
For $t > h_m$, and all $1 \leq i \leq m$, we have
\small
\begin{equation}\label{eq:link_substate_cropped_predictor_tilde_Yi}
\begin{split}
& Z_i(t) = e^{ \overline{A}_{ii} h_i } \Tilde{Y}_i(t-h_i) + \int_{t-h_i}^t  e^{ \overline{A}_{ii} (t-s) } \overline{\sigma}_i(s) dW_s \\
& + \int_{0}^t e^{ \overline{A}_{ii} (t-s) } \left( \sum_{j > i} \overline{A}_{ij} Z_j(s) + \sum_{j > i} \overline{B}_{ij} U_j(s-h_j) + \overline{r}_i(s) \right)ds.
\end{split}
\end{equation}
\normalsize
\end{lemma}
\begin{proof}
    The proof is given in Appendix \ref{APP_LEM_LINK_SUBSTATE}. 
\end{proof}

The previous lemma is not sufficient to compute the minimal covariance of $Z_i$. The additional delays introduce terms like  $\int_{0}^t e^{ \overline{A}_{ii} (t-s) } \left( \sum_{j > i} \overline{A}_{ij} Z_j(s) \right) ds$ which hinders a direct separation of $Z_i$  into a controllable and an independent, non-controllable term. However, by leveraging the triangular structure of the dynamics, we can perform this separation recursively, starting with $Z_m$.
%\commentGV{Avec $Y_i$ ou avec $\Tilde{Y}_i$? Pour le moment $Y_i$ on verra si c'est réellement nécessaire.}
\begin{theorem}\label{thm:substate_separation_for_min_cov_induction}
    Let $t > h_m$, and all $1 \leq i \leq m$, we can rewrite the substate $Z_i$ as
\begin{equation}\label{eq:substate_separation_for_min_cov_induction}
\begin{split}
    %Z_i(t) & = e^{ \overline{A}_{ii} h_i } \Tilde{Y}_i(t-h_i) + R^{(U)}_i(t-h_{i+1}) \\
    Z_i(t) & = e^{ \overline{A}_{ii} h_i } \Tilde{Y}_i(t-h_i) + R^{(U)}_i(t-h_{i+1}) \\
    & + \int_{t-h_{i+1}}^t \Gamma_i(t-s) \overline{\sigma}(s) dW_s,
\end{split}
\end{equation}
where $ R^{(U)}_i$ is a $\mathcal{F}_{t}$-adapted process, dependent on the input $U$ and system parameters, and $\Gamma_i$ is a deterministic function taking values in $ \mathbb{R}^{N_i \times N}$, defined as
%dependant only of the parameters of the system. %(\commentGV{Attention à bien définir $h_{m+1}$}) 
%Additionally, the expression for $\Gamma_i$ can be obtained by induction with
%\begin{equation}\label{eq:expression_induction_Gamma_i}
%\begin{split}
%\Bigl( \Gamma_i(u) \Bigr)_l & \triangleq \int_{\max(u-h_i,0) }^{\min(u, h_m)} e^{ \overline{A}_{ii} (u-s) } \overline{G}_i(u) ds \\
%& + \delta_{li} \indic_{[0,h_i]}(u) \ e^{ \overline{A}_{ii} u } + \sum_{j=i+1}^m \Theta_{i,j}(u),
%\end{split}
%\end{equation}
\begin{equation}\label{eq:expression_induction_Gamma_i}
\begin{split}
\Gamma_i(u) & \triangleq \int_{\max(u-h_i,0) }^{\min(u, h_m)} e^{ \overline{A}_{ii} (u-s) } \overline{G}_i(u) ds \\
& + \indic_{[0,h_i]}(u) \ e^{ \overline{A}_{ii} u } P_i + \sum_{j=i+1}^m \Theta_{i,j}(u),
\end{split}
\end{equation}
where $P_i \in  \mathbb{R}^{N_i \times N}$ is the projection onto the subspace of $Z_i$ (that is the matrix such that $P_i Z = Z_i$),  %(\commentGV{Meilleure façon de dire/donner la formule? Il s'agit techniquement d'une "injection canonique" puisque l'espace d'arrivé est plus petit.}) 
and $\Theta_{i,j}(u) \in \mathbb{R}^{N_i \times N}$ is given by 
$$
\Theta_{i,j}(u) \triangleq \int_{\max(u-h_i,0 )}^{\min(u,h_j)}  e^{ \overline{A}_{ii} (u-s') }  \overline{A}_{ij}  \Gamma_j(s')  ds'.
$$
%\commentGV{ATTENTION IL FAUT DIFFERENCIER LES COMPOSANTES DE $\Gamma$ QUI MULTIPLIENT LES DIFFÉRENTS $\sigma_i$. Le $\delta$ est une solution provisoire mais pas juste en terme de dimension.}
\end{theorem}
\begin{proof}
The proof is done recursively. The initialization is proved with $h_{m+1} = 2 h_m$ by applying the results from \cite{velho_stabilization_2024} as the dynamic of $Z_m$ has only one delay. The rest of the proof is obtained by injecting the expression of $Z_j$ \eqref{eq:substate_separation_for_min_cov_induction} into equation \eqref{eq:link_substate_cropped_predictor_tilde_Yi}. A full proof with detailed computations can be found in Appendix~\ref{APP_THM_SEPARATION}   %\modifJA{XXX}.
\end{proof}
%\begin{remark}
    Expression \eqref{eq:substate_separation_for_min_cov_induction} for the substate $Z_i$ %\modifJA{Faire ref à l'équation}
    gives %an immediate
    a candidate for the structural minimal covariance, as we have
    %$$ \Sigma_{Z_i}(t) > \Sigma^{min}_i(t) \triangleq \int_{t-h_i}^t \Gamma_i(t-s) \sigma(s) \sigma(s)^T \Gamma_i(t-s)^T ds .$$
    \begin{equation}\label{eq:minimal_covariance_definition}
     \Sigma_{Z_i}(t) > \Sigma^{(i)}_{\min}(t) \triangleq \int_{t-h_i}^t \Gamma_i(t-s) \sigma(s) \sigma(s)^T \Gamma_i(t-s)^T ds .
     \end{equation}
%\end{remark}
To demonstrate that this bound is the minimal covariance, we need to show that there exist adapted controls capable of making $\Sigma_{Z_i}(t)$ approach this bound as closely as desired.

\subsection{Finite Time Controllability}

In what follows, we show that the previously derived bound \eqref{eq:minimal_covariance_definition}% \modifJA{Faire ref à l'équation. Donc definir $\Sigma_{min}$ avec une équation.} 
can be approached arbitrarily closely. To establish this result, we leverage the finite-time controllability of the process $\Tilde{Y}_i$ (see \cite[Theorem 13]{mahmudov_controllability_2001}) to make the variance of the term
$$
e^{ \overline{A}_{ii} h_i } \Tilde{Y}_i(t-h_i) + R^{(U)}_i(t-h_{i+1}) + \int_{t-h_{i+1}}^{t-h_i} \Gamma_i(t-s) \overline{\sigma}(s) dW_s
$$
as small as desired. More precisely, we have the following controllability theorem
\begin{theorem}\label{thm:controllability_above_minimal_covariance_bound}
Let $T>h_m$, $\Sigma_T \in \mathcal{S}^{++}_n$ and $Z_T \in \mathbb{R}^N$. There exists a control $U \in \mathcal{U}$ such that the solution $Z$ to \eqref{eq:multi_input_delayed_SDE_kalman_form} associated with $U$ verifies $\espE[Z(T)] = Z_T$ and
$\Sigma_{Z_i}(T) = \Sigma^{(i)}_T + \Sigma^{(i)}_{\min}(T)$.
\end{theorem}
\begin{proof}
The proof can be be done by induction. It is given in Appendix~\ref{APP_THM_CONTROLABILITY}. % \modifJA{Donner le numéro}.
% We proceed by induction, allowing us to determine  $R^{(U)}_i(t-h_{i+1})$ for each substate. This term depends only on system parameters and control terms established in previous steps of the induction. We then use the controllability of $\Tilde{Y}$ and the fact that $\Tilde{Y}_i(t-h_i)$ can anticipate $R^{(U)}_i(t-h_{i+1})$ slightly in advance to compensate it. A full proof of this theorem is provided in the Appendix.
\end{proof}
 Selecting $\Sigma_T$ arbitrarily small, Theorem \ref{thm:controllability_above_minimal_covariance_bound} implies that we can steer the state $Z$ to a final covariance that can be arbitrarily close to the limit $\Sigma_{\min}(T)$.
 However, the control developed in the proof of Theorem \ref{thm:controllability_above_minimal_covariance_bound} % \modifJA{Où a-t-il été explicité?} 
 for guiding the system is an open-loop stochastic control %\modifJA{Qu'entends-tu par ça?}
 , which presents challenges for numerical implementation \cite{touzi_introduction_2013}. Therefore, we need to explore the potential of addressing the steering issue with a feedback controller. Below, we show that a simple feedback controller can effectively stabilize the system in terms in mean, while also ensuring that the variance remains bounded.

\subsection{Stabilization of the PDE+SDE}
To achieve infinite time control of the mean-variance, we can design a simple controller that exponentially stabilizes the mean of the system to zero while ensuring the covariance remains bounded.
\begin{theorem}\label{thm:stabilization_system_feedback}
Define the feedback controller $U_i(t) = - K_i Y_i(t),$
where $Y_i$ is given by equation~\eqref{eq:Def_Artstein_predictor_substate_Zi}, and $K_i$  is  such that $H_i \triangleq \overline{A}_{ii} - \Tilde{B}_{ii} K_i$ is Hurwitz. Let $\Vin$ be the control law $\Vin(t) = \VPDE(t) + \Veff(t)$, with $\VPDE$ as defined in \eqref{eq:definition_expression_backstep_controller} and $\Veff$ as defined in \eqref{eq:def_of_Veff_in_terms_of_VSDE} with $\VSDE$ set as $\VSDE^{(i)} \triangleq U_{m-i}$. 
%tracks $U(t)$ through the PDE $V_{in}(t) = \VPDE(t) + \Veff(t)$ \modifJA{Mieux formuler ça. }
Then the control $\Vin$ steers the expectation of the PDE-SDE system state to zero while keeping the variance bounded, as in \eqref{eq:stabilization_mean_expectation}.
%stabilizes the original system \eqref{eq:coupled_PDE_SDE}. \modifJA{Au sens de la moyenne tout en gardant une variance bornée} 
% That is, there exist $C>0$, (dependent on the parameters of the system and $K$) and $\nu>0$ (dependent on the eigenvalues of $H_i$), such that, for any initial conditions $(X_0,u_0,v_0)$, for all $x$ in $[0,1]$ and $t>0$:
% \begin{align*}
%    &  \Vert \espE[X(t)] \Vert + \Vert \espE[u(x,t)] \Vert + \Vert \espE[v(x,t)] \Vert \leq C e^{- \nu t} \times\\
%    & \quad \times \bigl (\Vert \espE[X_0] \Vert + \Vert  \espE [u_0] \Vert + \Vert  \espE [v_0] \Vert \bigr ), \\
%    &  \espE \left[ \Vert X(t)\Vert^2 \right]  + \espE \left[ \Vert u(x,t) \Vert^2 \right] + \espE \left[ \Vert v(x,t) \Vert^2 \right] \leq C
% \end{align*}
% \modifJA{En fait, ce critère devrait apparaître dès la section objectif.}
 \end{theorem}
 \begin{proof}
The proof of the Theorem can be found in Appendix \ref{APP_THM_STABILIZATION}. % \modifJA{XXX}.
\end{proof}

%\subsection{Finite time Covariance Steering}

\section{Optimal Control of the SDE}\label{OPT}

In the previous section, we established that while open-loop controls can steer variance to a minimal threshold, they are not practical to implement. %and do not guarantee consistently low variance throughout time. We also developed a feedback controller to stabilize the interconnected PDE-SDE system, though this approach does not necessarily achieve minimized variance.

To address this, we turn to optimal control techniques that minimize a linear-quadratic (LQ) cost functional, making them ideal candidates for maintaining low variance. The LQ cost functional is given by:
\begin{equation}\label{eq:def_quadratic_optimal_cost_2}
\begin{split}
    & J_R(\Veff,Z) \triangleq  \espE \left[ \int_{h_m}^T Z(t)^T Q(t) Z(t) dt \right] \\
    & \quad + \espE \left[\int_0^{T-h_m} U(t)^T R(t) U(t) dt \right],
\end{split}
\end{equation}

where $Q \in L^\infty([0,T], \mathcal{S}_n^+)$ and $R \in L^\infty([0,T], \mathcal{S}_n^{++})$ balance state variance and control effort over time.

Explicit solutions for this LQ problem with multiple input delays and random coefficients are complex and, to the extend of our knowledge, unresolved for the general case. Under additional assumptions, we can derive results for particular cases. We analyze the optimal variance achieved by minimizing $J_R$ and demonstrate that, given certain controllability conditions, the optimal LQ control can keep the state variance arbitrarily close to the minimal achievable variance throughout time as we make the parameter $R$ go to zero.

%where $Q \in L^\infty([0,T], \mathcal{S}_n^+)$ and $R \in L^\infty([0,T], \mathcal{S}_n^{++})$. This functional cost penalizes the weighted variance along the trajectory and the control effort. The matrices are selected based on the desired trade-off between penalizing the state and the control effort over time.
%\textcolor{blue}{J'ai décidé de ne pas mettre de coût final pour gagner un peu d'espace. Commentaire?}
% Our problem thus states:
% \begin{equation}\label{eq:problem_min_cost_X}
%     \min_{U \in \mathcal{U}} J_R(U,Z) , \quad \text{$Z$ solves SDE \eqref{eq:multi_input_delayed_SDE_kalman_form}.}
% \end{equation}
% We use existing results found in the literature to compute the controller.

%So we can find some papers that present results on the control of multi input delayed SDEs however they may not be that straightforward to use in our case, no clear explicit theorem has been given (to the extent of our knowledge) to this case.

\subsection{Optimal control of the equivalent SDE}

% In what follows we provide two cases where an explicit expression of the optimal control minimizing the cost  can be found.

% \begin{itemize}
%     \item The case with one delay has been treated in the preliminary study with one delay only in \cite{velho_stabilization_2024}.
%     \item The case with $M=0$ in system \eqref{eq:coupled_PDE_SDE}, that is the system with no counter actuation of the SDE into the PDE. This means that the equivalent delayed SDE does not have random drift, and we can apply recent results from optimal LQ control of stochastic systems with deterministic coefficients and with multi input delays described in \cite{wang_lq_2013, wang_optimal_2021}. Note however that these controls are not directly implementable as there seems to be lacking a numerical implementation for solving the presented Riccati equation. Some advances have been done recently for discrete time systems as in \cite{wang_lqr_2024}, but the continuous time needs to be treated yet.  
% \end{itemize}

In the following, we present two cases where an explicit form of the optimal control minimizing the cost \eqref{eq:def_quadratic_optimal_cost_2} can be derived:

\begin{itemize} 
\item The single-delay case was examined in our previous study \cite{velho_stabilization_2024}, where the optimal control for one delay was explicitly determined. 
\item For the case $M=0$ in system \eqref{eq:coupled_PDE_SDE}—where there is no feedback from the SDE into the PDE—the equivalent delayed SDE has no random drift, allowing us to use results from optimal LQ control of stochastic systems with deterministic coefficients and multiple input delays \cite{wang_lq_2013, wang_optimal_2021}. However, these controls currently lack a direct numerical implementation for solving the associated Riccati equation in continuous time. While progress has been made recently for discrete-time systems \cite{wang_lqr_2024}, continuous-time solutions remain an open area for future work. 
\end{itemize}

\subsection{Attaining the minimal variance}

In what follows, we study the variance associated to the optimal control minimizing \eqref{eq:def_quadratic_optimal_cost_2}.

We define the minimal weigthed variance by
$$ V^{(i)}_{\min,Q}(t) \triangleq \int_{t-h_i}^t \sigma(s)^T \Gamma_i(t-s)^T Q(t) \Gamma_i(t-s) \sigma(s)  ds,$$
and the minimum cost over $[0,T]$ by
$$
J_{\min,Q}^{(i)} \triangleq \int_{h_m}^T V^{(i)}_{\min,Q}(t) dt, 
$$

To prove that the optimal control achieves minimal variance in each substate as the $R$ goes to zero, we will first establish, under some additional controllability assumptions on $A$ and $B$, the existence of a sequence of controls that approach this minimal variance. Next, we show that the optimal control must outperform that sequence for certain as the penalization $R \rightarrow 0$. Note that assuming full actuation is not particularly restricting, in that, if this assumption is not satisfied, one can prove the existence of random states that can not be reached by any open-loop control, see, e.g., \cite{wang_exact_2017, velho_mean-covariance_2023}.

%\modifJA{Expliquer pourquoi on a besoin d'hypothèses et expliciter leur conservatisme}

\begin{assumption}\label{asm:assumption_n_equal_m}
    We assume that $m=n$, therefore the matrices $\overline{B}_{ii}$ are non zero real numbers.
\end{assumption}

\begin{assumption}\label{asm:R_is_equal_to_identity}
    To enhance clarity, we assume that the penalization of the cost $R$ is proportional to the identity $R = \rho I$. We denote $J_\rho(U)$ the cost associated to the regularization parameter $\rho$ of the control $U$. Our results would also hold for more general matrices $R$.
\end{assumption}

\begin{theorem}\label{thm:sequence_of_controls_to_minimal_variance}
    Under assumptions \ref{asm:assumption_n_equal_m} and \ref{asm:R_is_equal_to_identity}, there exists a sequence of adapted controllers $U_n \in \mathcal{U}$ such that the sequence of substates $Z_i^{(n)}$ verify
    $$
    \int_{h_i}^T  V_{Z_i^{(n)}}(t) dt  \rightarrow \int_{h_i}^T  V^{(i)}_{\min,Q}(t) dt
    $$
\end{theorem}
\begin{proof}
    The proof can be found in Appendix \ref{APP_THM_SEQUENCE_TO_MINIMAL} %\modifJA{XXX}.
\end{proof}

Implementing these adapted controllers may present practical challenges. However, they demonstrate that the previously obtained optimal controls, which are easily implementable, make the optimal cost converge to the minimal weighted variance as the penalization parameter $R$ goes to zero.
%\modifJA{Est-ce vraiment le contrôle qui converge vers la variance minimale? N'est ce pas le système en BF?}
\begin{theorem}\label{thm:optimal_control_converging_to_min_var}
    Let $U_\rho^*$ be the optimal control minimizing $J_\rho$, then
$$
J^{(i)}_\rho( U_\rho^*)  \underset{\rho \rightarrow 0}{\rightarrow}  \int_{h_m}^T  V^{(i)}_{\min,Q}(t) dt.
$$
\end{theorem}
\begin{proof}
    The proof follows similar arguments to those used in the proof of \cite[Lemma 10]{velho_mean-covariance_2023}. Let us consider $U_n$ as defined in %\modifJA{??? }
    \ref{thm:sequence_of_controls_to_minimal_variance}, and choose $\rho_n$ such that $ \rho_n \espE \left[\int_0^{T-h} U_n(t)^T U_n(t) dt \right] \rightarrow 0$. It implies that $J_{\rho_n}^{(i)}(U_n) \rightarrow J_{\min}^{(i)}$. Combined with the optimality of $U^*_{\rho_n}$, it implies that $J_{\rho_n}^{(i)}(U^*_{\rho_n}) \rightarrow J_{\min}^{(i)}$.
\end{proof}

\section{Simulations}\label{SIM}

To showcase the efficiency of our stabilizing method in a multi-input delayed system, we apply the control strategy described in Theorem \ref{thm:stabilization_system_feedback} to an academic example. 
We consider the unstable system
\small
% $$
% \left\{
% \begin{array}{l}
%  d\left( \begin{array}{c} x_1(t) \\ x_2(t) \end{array} \right) = \left( \begin{array}{cc} a & a \\ 0 & a  \end{array} \right) \left( \begin{array}{c} x_1(t) \\ x_2(t) \end{array} \right)dt \\
%  \hspace{5em}+ \left( \begin{array}{cc} b & -b \\ 0 & b  \end{array} \right) \left( \begin{array}{c} U_1(t-h_1) \\ U_2(t-h_2) \end{array} \right)dt \\
%  \hspace{5em} + \left( \int_{t-h_2}^t e^{-\theta (t-s) } \sigma dW_s \right) dt + \sigma dW_t 
% \end{array} \right.
% $$
\begin{align*}
 & d\left( \begin{array}{c} x_1(t) \\ x_2(t) \end{array} \right)  = \left( \begin{array}{cc} a & a \\ 0 & a  \end{array} \right) \left( \begin{array}{c} x_1(t) \\ x_2(t) \end{array} \right)dt \\
& \hspace{3em}  + \left( \begin{array}{cc} b & -b \\ 0 & b  \end{array} \right) \left( \begin{array}{c} U_1(t-h_1) \\ U_2(t-h_2) \end{array} \right)dt \\
& \hspace{3em}  + \left( \int_{t-h_2}^t e^{-\theta (t-s) } \left( \begin{array}{c} \sigma \\ \sigma \end{array} \right) dW_s \right) dt + \left( \begin{array}{c} \sigma \\ \sigma \end{array} \right)dW_t 
\end{align*}
\normalsize
with $a = 0.4$, $b=2$, $h_1 = 0.5$, $h_2 = 1$, $\theta = 0.2$, $\sigma = 0.3$.

\begin{figure}[ht!]
\centering
  \includegraphics[width=0.90\linewidth]{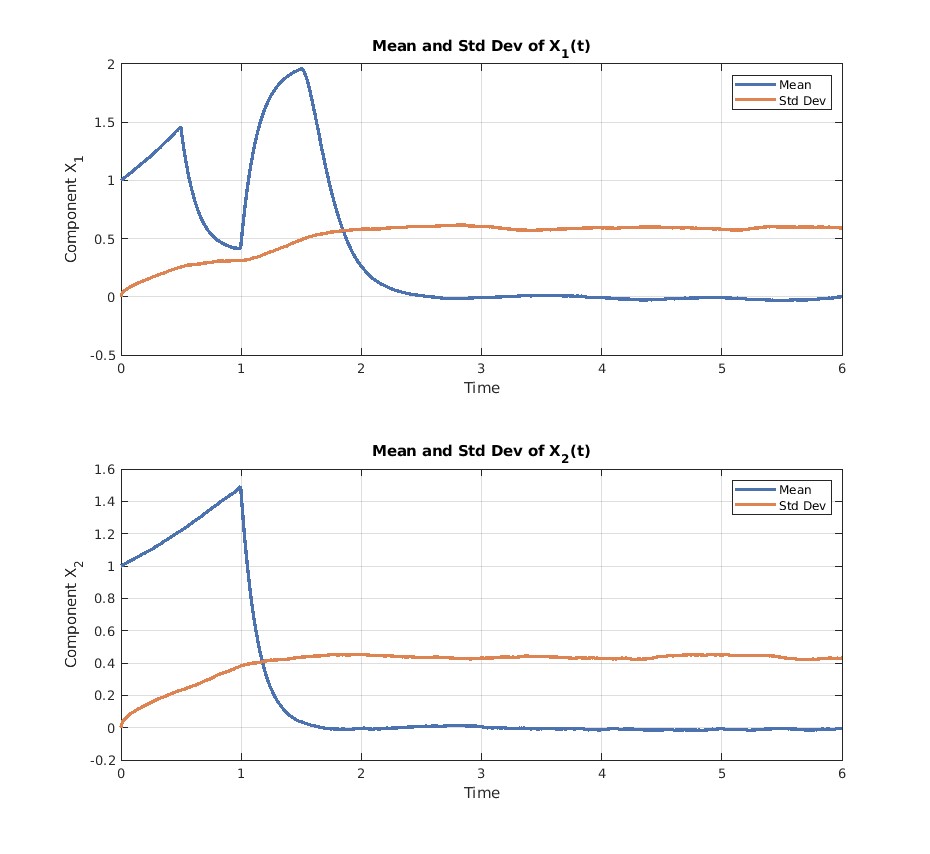}
  \caption{\centering Mean and deviation (square root of the variance) of the two components of the SDE state are depicted. The values were computed through a Monte-Carlo approach.}% \textcolor{red}{mettre LQ et non pas LQT}}
  \label{fig:plot_SDE_mean_var_closed_open}
\end{figure}
Figure \ref{fig:plot_SDE_mean_var_closed_open} shows that, despite initial spikes due to the two input delays, the controller successfully stabilizes the state while keeping the deviation bounded.

%\section{Comments on the Robustness}\label{ROB}
%\input{07_Robustness_Comments}
%\modifJA{On peut virer cette section.}

\section{Conclusion}\label{CCL}

In this paper, we have studied the controllability of bidirectionally interconnected PDE+SDE systems, where the PDE is a general system of $n + m$ linear first-order equations. We demonstrated that this coupling introduces delays and a stochastic drift in the SDE, creating unique challenges for control. Two control strategies were proposed: one for stabilization in mean with bounded state variance, and another that minimizes the SDE variance, enhancing robustness against noise. Under additional controllability assumptions, we showed that the optimal control can approximate the structural minimum variance achievable.
We plan to explore further avenues for research, including the generalization of the PDE to a parabolic system, as well as the extension of the SDE framework to incorporate multiplicative noise. We also plan to investigate the potential benefits of employing alternate backstepping techniques used successfully in PDE+ODE systems, such as the multi-step approach \cite{irscheid_observer_2021}, or dimensional reduction techniques \cite{morris_controller_2020}.

%introduced two approaches for the control of bidirectionally coupled PDE+SDE systems. These two approaches utilize a combination of classical PDE+ODE tools and stochastic control techniques. One approach achieves stabilization in mean with bounded variance of the states, while the other minimizes the SDE variance to ensure additional robustness against noise.
%We plan to explore further avenues for research, including the generalization of the PDE to a system of $n \times m$ linear first-order equations, as well as the extension of the SDE framework to incorporate multiplicative noise. \revAns{We also plan to investigate the potential benefits of employing alternate backstepping techniques used successfully in PDE+ODE systems, such as the multi-step approach \cite{irscheid2023output}.

%\modifJA{Attention a la biblio. Certaines refs sont bizarres (Touzi et notre article)}
\bibliographystyle{ieeetr}
\bibliography{references_clean}

\section{Appendix}\label{A1}

\subsection{Proof of Lemma \ref{lem:track_VSDE_with_Veff_and_formula_beta}}\label{APP_LEM_TRACK_NEW}

\begin{proof}

We prove that there exists matrix functions $F_i$ and $G_i$ such that for all $1 \leq i \leq m$ we have
\begin{equation}\label{eq:formula_for_beta_induction_new}
\begin{split}
    \beta_i(t,x) & = \beta_i \left(t- \frac{1-x}{\mu_i} , 1 \right) \\
    & \quad +  \sum_{j=i+1}^m \int_{\frac{1-x}{\mu_j}}^{\frac{1-x}{\mu_i}} \left( F_i \right)_j(x,s) \beta_j(t-s,1) ds, \\
    & \quad + \int_{t - \frac{1-x}{\mu_i}}^t G_i(x,t-s) \sigma(s) dW_s,
\end{split}
\end{equation}
Then, by (recursively) setting 
\begin{align*}
\Veff^{(i)}(t) & = \VSDE^{(i)}\left(t + \frac{1}{\mu_i} - \frac{1}{\mu_1} \right) \\
& -  \sum_{j=i+1}^m \int_{\frac{1}{\mu_j}}^{\frac{1}{\mu_i}} \left( F_i \right)_j(0,s) \Veff^{(j)}\left(t + \frac{1}{\mu_i} - s\right) ds
\end{align*}
we obtain that
\begin{itemize}
    \item $\Veff^{(i)}(t)$ is $\mathcal{F}_{t + \frac{1}{\mu_i} -\frac{1}{\mu_1}}$-adapted.
    \item $\beta_i(t,0) = \VSDE \left(t- \frac{1}{\mu_1} \right) + \int_{t - \frac{1}{\mu_i}}^t G_i(0,t-s) \sigma(s) dW_s,$
\end{itemize}
which concludes our proof.
% \commentGV{
% NO, controller is not causal! We need to prove that 
% \begin{equation}\label{eq:formula_for_beta_induction_new_causal}
% \begin{split}
%     \beta_i(t,x) & = \beta_i \left(t- \frac{1-x}{\mu_i} , 1 \right) \\
%     & \quad +  \sum_{j=i+1}^m \int_0^{\frac{1-x}{\mu_i}} \left( F_i \right)_j(x,s) \beta_j(t-s \textcolor{red}{ - \frac{1-x}{\mu_i}} ,1) ds, \\
%     & \quad + \int_{t - \frac{1-x}{\mu_i}}^t G_i(x,t-s) \sigma(s) dW_s,
% \end{split}
% \end{equation}
% En fait en allant dans ce sens la j'ai l'impression qu'on essaie de compenser des signaux qui partent après... ou alors calculer en fonction de $ \beta_i \left(t+ \frac{x}{\mu_i} , 0 \right)$?
% Peut être quelque chose comme
% \begin{equation}\label{eq:formula_for_beta_induction_new_causalV3}
% \begin{split}
%     \beta_i(t,x) & = \VSDE^{(i)} \left(t-\delta_i + \frac{x}{\mu_i}\right) \\
%     & \quad +  \sum_{j=i+1}^m \int_0^{\frac{x}{\mu_i}} \left( F_i \right)_j(x,s) \VSDE^{(j)} \left(t+s -\delta_i \right) ds, \\
%     & \quad + \int_{t - \frac{1-x}{\mu_i}}^t G_i(x,t-s) \sigma(s) dW_s,
% \end{split}
% \end{equation}
% OR
% \begin{equation}\label{eq:formula_for_beta_induction_new_causalV2}
% \begin{split}
%     \beta_i(t,x) & = \beta_i \left(t- \frac{1-x}{\mu_i} , 1 \right) \\
%     & \quad +  \sum_{j=i+1}^m \int_{\textcolor{red}{ \frac{1-x}{\mu_j}}}^{\frac{1-x}{\mu_i}} \left( F_i \right)_j(x,s) \beta_j(t-s \textcolor{red}{ - \frac{1-x}{\mu_j}} ,1) ds, \\
%     & \quad + \int_{t - \frac{1-x}{\mu_i}}^t G_i(x,t-s) \sigma(s) dW_s,
% \end{split}
% \end{equation}
% }

We now prove \eqref{eq:formula_for_beta_induction_new} by induction.\\
\noindent \textbf{Initialization:} The function $\beta_m$ verifies the following SPDE:
$$
d \beta_m(t,x) - \mu_m (\beta_m)_x(t,x) = \gamma_\beta^{(m)}(x) \sigma(t) dW_t.
$$
Therefore, its explicit expression is given by the stochastic convolution \cite{hairer_introduction_2009}
\begin{equation}\label{eq:formula_for_beta_m_new_induction}
\begin{split}
    \beta_m(t,x) & = \Veff^{(m)} \left(t - \frac{1-x}{\mu_m} \right) \\
    & \quad + \int_{t - \frac{1-x}{\mu_m}}^t \gamma_\beta^{(m)}(x + \mu_m(t-s)) \sigma(s) dW_s,
\end{split}
\end{equation}
and equation \eqref{eq:formula_for_beta_induction_new} is verified by setting $G_m$ as
\begin{align*}
G_m(x,t-s) \triangleq \gamma_\beta^{(m)}(x + \mu_m(t-s)).
\end{align*}

\noindent \textbf{Induction:} Let $1 \leq i \leq m-1$ be an integer and assume that~\eqref{eq:induction_hypo_lemma_tracking_delayed_SDE_beta} is verified for all $j\geq i$. The state $\beta_i$ verifies the following SPDE
\begin{align*}
d \beta_i(t,x) - \mu_i (\beta_i)_x(t,x) & = \left( \sum_{j=i+1}^m \omega_{i,j}(x) \beta_j(t,x)   \right) dt \\
& \quad \ + \gamma_\beta^{(i)}(x) \sigma(t) dW_t.
\end{align*}
The explicit expression of $\beta_i$ is therefore given by
\small
\begin{align*}
& \beta_i(t,x) = \Veff^{(i)} \left(t - \frac{1-x}{\mu_i} \right) \\
& + \int_{t - \frac{1-x}{\mu_i}}^t \left( \sum_{j=i+1}^m \omega_{i,j}(x + \mu_i(t-s)) \beta_j(s,x + \mu_i(t-s))   \right) ds \\
& + \int_{t - \frac{1-x}{\mu_i}}^t \gamma_\beta^{(i)}(x + \mu_i(t-s)) \sigma(s) dW_s.
\end{align*}
\normalsize
our induction hypothesis states that 
\begin{align*}
 & \beta_j(s,x+\mu_i (t-s) ) = \beta_j \left(s- \frac{1-x-\mu_i (t-s)}{\mu_j} , 1 \right) \\
    &  +  \sum_{k=j+1}^m \int_{\frac{1-x-\mu_i (t-s)}{\mu_k}}^{\frac{1-x-\mu_i (t-s)}{\mu_j}}  F_{jk}(x+\mu_i (t-s),\tau) \beta_k(s-\tau,1) d\tau \\
    & + \int_{s - \frac{1-x-\mu_i (t-s)}{\mu_j}}^{s} G_j(x+\mu_i (t-s),s-\tau) \sigma(\tau) dW_\tau,
\end{align*}
we inject this formula into the expression of $\beta_i(t,x)$ and obtain
\small
\begin{align*}
& \beta_i(t,x) = \Veff^{(i)} \left(t - \frac{1-x}{\mu_i} \right) \\
& + \int_{t - \frac{1-x}{\mu_i}}^t \Biggl[ \sum_{j=i+1}^m \omega_{i,j}(x + \mu_i(t-s)) \times \\
& \hspace{5em} \beta_j \left(s- \frac{1-x-\mu_i (t-s)}{\mu_j} , 1 \right)  \Biggr] ds\\
& + \int_{t - \frac{1-x}{\mu_i}}^t \Biggl[ \sum_{j=i+1}^m \omega_{i,j}(x + \mu_i(t-s)) \times \\
& \quad  \sum_{k=j+1}^m \int_{\frac{1-x-\mu_i (t-s)}{\mu_k}}^{\frac{1-x-\mu_i (t-s)}{\mu_j}} \left( F_j \right)_k(x+\mu_i (t-s),\tau) \beta_k(s-\tau,1) d\tau \Biggr] ds \\
& + \int_{t - \frac{1-x}{\mu_i}}^t \Biggl[ \sum_{j=i+1}^m \omega_{i,j}(x + \mu_i(t-s)) \times \\
& \quad \int_{s - \frac{1-x-\mu_i (t-s)}{\mu_j}}^{s} G_j(x+\mu_i (t-s),s-\tau) \sigma(\tau) dW_\tau \Biggr] ds \\
& + \int_{t - \frac{1-x}{\mu_i}}^t \gamma_\beta^{(i)}(x + \mu_i(t-s)) \sigma(s) dW_s.
\end{align*}
\normalsize
We permute the sums and the integrals with the Fubini theorem on the term containing $\left( F_j \right)_k$, we also permute the integrals on the stochastic integral with the stochastic Fubini theorem \cite[Chapter 4, Theorem 45]{protter_general_2005}. More details on this exchange can be found in the supplementary material \ref{EXCHANGE_INT}. This, combined with a change of variable, leads to %equation \eqref{eq:formula_for_beta_induction_new} with

\begin{align*}
& \beta_i(t,x) = \Veff^{(i)} \left(t - \frac{1-x}{\mu_i} \right) \\
& + \sum_{j=i+1}^m \int_{\frac{1-x}{\mu_j}}^{\frac{1-x}{\mu_i}} \Hat{F}_{ij}(x,s) \beta_j(t-s,1) ds \\
& + \sum_{k=i+2}^m \int_{\frac{1-x}{\mu_k}}^{\frac{1-x}{\mu_i}} \Tilde{F}_{ik}(x,\tau) \beta_k(t-\tau,1) d\tau \\
& + \int_{t - \frac{1-x}{\mu_i}}^t G_i(x,t-s) \sigma(s) dW_s 
\end{align*}
with 
\small
\begin{equation}\label{eq:definition_2_F_multi_Veff}
\begin{split}
  \Hat{F}_{ij}(x,s) & \triangleq \omega_{ij}\left(x + \frac{\mu_i \mu_j}{\mu_j - \mu_i} \left( s - \frac{1-x}{\mu_j} \right) \right) \\
   \Tilde{F}_{ik}(x,\tau) & \triangleq \sum_{j = i+1}^{k-1} \int_{\underline{s}_{ij}(x,\tau)}^{\overline{s}_{ik}(x,\tau)} \omega_{i,j}(x + \mu_i s) F_{jk}(x+\mu_i s, \tau - s) ds\\
   \underline{s}_{ij}(x,\tau) & \triangleq \max \left( 0, \frac{\mu_j \tau - (1-x)}{\mu_j - \mu_i} \right) \\
   \overline{s}_{ik}(x,\tau) & \triangleq \frac{\mu_k \tau - (1-x)}{\mu_k - \mu_i}
\end{split}
\end{equation}
\normalsize
and $G_i$ given by \eqref{eq:def_of_G_state_feedback_noise}. We obtain  \eqref{eq:formula_for_beta_induction_new} by setting
\begin{equation}\label{eq:definition_1_F_multi_Veff}
  F_{ij}(x,s) \triangleq \Hat{F}_{ij}(x,s) + \delta_{j \neq i+1} \Tilde{F}_{ij}(x,s)
\end{equation}

% that we can re-write as
% \begin{align*}
% & \beta_i(t,x) = \Veff^{(i)} \left(t - \frac{1-x}{\mu_i} \right) \\
% & + \int_{0}^{\frac{1-x}{\mu_i}} \left( \sum_{j=i+1}^m \omega_{i,j}(x + \mu_i s) \beta_j(t-s,x + \mu_i s )   \right) ds \\
% & + \int_{t - \frac{1-x}{\mu_i}}^t \gamma_\beta^{(i)}(x + \mu_i(t-s)) \sigma(s) dW_s.
% \end{align*}
% our induction hypothesis states that 
% \begin{align*}
%  & \beta_i(t-s,x+\mu_i s) = \beta_i \left(t-s- \frac{1-x+\mu_i s}{\mu_i} , 1 \right) \\
%     & \quad +  \sum_{j=i+1}^m \int_0^{\frac{1-x+\mu_i s}{\mu_i}} \left( F_i \right)_j(x+\mu_i s,\tau) \beta_j(t-s-\tau,1) d\tau, \\
%     & \quad + \int_{t-s - \frac{1-x+\mu_i s}{\mu_i}}^{t-s} G_i(x+\mu_i s,t-s-\tau) \sigma(\tau) dW_\tau,
% \end{align*}

\end{proof}

\subsection{Proof of Theorem \ref{thm:Kalman_canonical_extended}}\label{APP_THM_KALMANN}
 Theorem \ref{thm:Kalman_canonical_extended} is an extension of the Kalman's decomposition theorem~\cite{dahleh_mit_2011}  %\modifJA{Donner ref.}
\begin{theorem}\label{thm:Kalmann_canonical_base}
    Let $A \in \mathbb{R}^{n \times n} $ and $B \in \mathbb{R}^{n \times m} $ be two matrices, and let $X$ be defined such that
    $$
    \dot{X}(t) = AX(t) + BU(t).
    $$
    Then, there exists an invertible matrix $M$ such that $Z(t) \triangleq M X(t)$ verifies
$$
\dot{Z}(t) = \left( \begin{array}{cc} \overline{A}_{11} & \overline{A}_{12} \\ 0 & \overline{A}_{22} \end{array} \right) Z(t) + \left( \begin{array}{c} \overline{B}_{11} \\ 0 \end{array} \right) U(t),
$$
with $\overline{A} = M A M^{-1}$, $\left( \begin{array}{c} \overline{B}_{11} \\ 0 \end{array} \right) = M B$, and $(\overline{A}_{11} , \overline{B}_{11}) $ controllable. %\modifJA{Mettre les dépendances en temps? Ne faut-il pas des hypothèses sur les matrices $A$ et $B$? Si $B=0$, est-ce que ça marche ?}
\end{theorem}
We now prove the extension presented in Theorem \ref{thm:Kalman_canonical_extended}.

\begin{proof}
We proceed by induction on $m$. For $m=1$, Theorem \ref{thm:Kalman_canonical_extended} is true without any change of variable. Suppose now that the property holds for any $m' < m$. For the sake of clarity, we assume that $X$ follows the deterministic and non-delayed dynamic
$$
\dot{X}(t) = A X(t) + \sum_{i=1}^m B_i U_i(t),
$$
with $(A, (B_1, ... , B_m) )$ controllable. Kalman's decomposition theorem \ref{thm:Kalmann_canonical_base} applied to matrices $A$ and $B_1$ states that there exists a matrix $M_1$ such that $Z^{(1)} \triangleq M_1 X$ verifies
\small
\begin{align*}
& \dot{Z}^{(1)}(t)  = \left( \begin{array}{c | c} \overline{A}^{(1)}_{11} & \left( \overline{A}^{(1)}_{12}, ..., \overline{A}^{(1)}_{1m}  \right) \\ \hline \left( \begin{array}{c} 0 \\ \vdots \\ 0 \end{array} \right) & \overline{A}^{(1)}_2 \end{array} \right) \left( \begin{array}{c} Z^{(1)}_1(t) \\ \hline  Z^{(1)}_2(t) \\ \vdots \\ Z^{(1)}_m(t) \end{array} \right) \\
& \ +  \left( \begin{array}{c | c} \overline{B}^{(1)}_{11} & \left( \overline{B}^{(1)}_{12}, ..., \overline{B}^{(1)}_{1m}  \right) \\ \hline \left( \begin{array}{c} 0 \\ \vdots \\ 0 \end{array} \right) & \overline{B}^{(1)}_2 \end{array} \right) \left( \begin{array}{c} U_1(t) \\ \hline  U_2(t) \\ \vdots \\ U_m(t) \end{array} \right),
\end{align*}
\normalsize
with $(\overline{A}^{(1)}_{11}, \overline{B}^{(1)}_{11}) $ controllable. The change of variable is invertible% \modifJA{remplace tous les reversible par invertible}
, and $X$ can be steered to any state, therefore the substate $\left( \begin{array}{c} Z^{(1)}_2(t) \\ \vdots \\ Z^{(1)}_m(t) \end{array} \right)$ can also be steered to any state. This implies that the pair $( \overline{A}^{(1)}_2, \overline{B}^{(1)}_2)$ is controllable. We can apply our induction hypothesis to the dynamics of the substate with $m' = m-1$, and obtain a invertible matrix $M_2$ such that, if $\left( \begin{array}{c} Z^{(2)}_2(t) \\ \vdots \\ Z^{(2)}_m(t) \end{array} \right) \triangleq M_2 \left( \begin{array}{c} Z^{(1)}_2(t) \\ \vdots \\ Z^{(1)}_m(t) \end{array} \right)$, then
\begin{align*}
    \left( \begin{array}{c} \dot{Z}^{(2)}_2(t) \\ \vdots \\ \dot{Z}^{(2)}_m(t) \end{array} \right) & = \left( \begin{array}{ccc} \overline{A}_{22} & &  \overline{A}_{ij} \\  & \ddots & \\ (0) & & \overline{A}_{mm} \end{array} \right)  \left( \begin{array}{c} Z^{(2)}_2(t) \\ \vdots \\ Z^{(2)}_m(t) \end{array} \right) \\ 
    & + \left( \begin{array}{ccc} \overline{B}_{22} & & \overline{B}_{ij} \\  & \ddots & \\ (0) & & \overline{B}_{mm} \end{array} \right) \left( \begin{array}{c} U_2(t) \\ \vdots \\ U_m(t) \end{array} \right),
\end{align*}
with the pairs $(\overline{A}_{ii},\overline{B}_{ii})$ controllable for all $i \in [\![2,m]\!] $. Therefore, by setting 
$$
M \triangleq M_1 \left( \begin{array}{c | c}  1 & 0_{1 \times m} \\ \hline 0_{m \times 1} & M_2  \end{array} \right)
$$
we obtain a invertible change of variable such that $Z \triangleq M X$ verifies the dynamic 
\begin{align*}
    \left( \begin{array}{c} \dot{Z}_1(t) \\ \vdots \\ \dot{Z}_m(t) \end{array} \right) & = \left( \begin{array}{ccc} \overline{A}_{11} & &  \overline{A}_{ij} \\  & \ddots & \\ (0) & & \overline{A}_{mm} \end{array} \right)  \left( \begin{array}{c} Z_1(t) \\ \vdots \\ Z_m(t) \end{array} \right) \\ 
    & + \left( \begin{array}{ccc} \overline{B}_{11} & & \overline{B}_{ij} \\  & \ddots & \\ (0) & & \overline{B}_{mm} \end{array} \right) \left( \begin{array}{c} U_1(t) \\ \vdots \\ U_m(t) \end{array} \right),
\end{align*}
with the pairs $(\overline{A}_{ii},\overline{B}_{ii})$ controllable for all $i \in [\![1,m]\!] $. 
We can apply the transformation to the delayed stochastic dynamic given by \eqref{eq:equivalent_delayed_SDE} and obtain equation \eqref{eq:multi_input_delayed_SDE_kalman_form}.
By induction, this proves the extended Kalman decomposition for all $m$.
\end{proof}

\subsection{Proof of Lemma \ref{lem:link_substate_cropped_predictor_tilde_Yi}}\label{APP_LEM_LINK_SUBSTATE}
\begin{proof}
%We adapt the formula outlined in \cite{velho_mean-covariance_2023} for the link between the Artstein predictor and the original state. 
Let us re-write the dynamic of the substate $Z_i$ as
$$
dZ_i(t) = \left( \overline{A}_{ii} Z_i(t) + \overline{B}_{ii} U_i(t-h_i) + d_i(t) \right) dt + \overline{\sigma}_i(t) dW_t.
$$
with $d_i(t) \triangleq \sum_{j > i} \overline{A}_{ij} Z_j(t) + \sum_{j > i} \overline{B}_{ij} U_j(t-h_j) + \overline{r}_i(t)$. From \cite{velho_mean-covariance_2023} we have that
\begin{equation}\label{eq:appendix_link_artstein_original_substate_Yi}
\begin{split}
& Z_i(t) = e^{ \overline{A}_{ii} h_i } Y_i(t-h_i) + \int_{t-h_i}^t  e^{ \overline{A}_{ii} (t-s) } \overline{\sigma}_i(s) dW_s \\
& +  \int_{t-h_i}^t  e^{ \overline{A}_{ii} (t-s) } d_i(s) ds.
\end{split}
\end{equation}
Using the Itô formula, we obtain
$$
Y_i(t) = \Tilde{Y}_i(t) + \int_{0}^t e^{ \overline{A}_{ii} (t-s) } d_i(s) ds.
$$
By injecting the last formula into equation \eqref{eq:appendix_link_artstein_original_substate_Yi}, we obtain equation \eqref{eq:link_substate_cropped_predictor_tilde_Yi}.
\end{proof}

\subsection{Proof of Theorem \ref{thm:substate_separation_for_min_cov_induction}}\label{APP_THM_SEPARATION}
\begin{proof}
We proceed by induction, starting from the case $i=m$, where the substate $Z_m$ only has one delay and can be treated with existing results. For clarity, we set $h_{m+1} \triangleq 2h_m$, but the proof can be done for any set $h_{m+1} > h_{m}$.\\

\noindent\textbf{Initialization:}
Let $i=m$, equation \eqref{eq:link_substate_cropped_predictor_tilde_Yi} yields 
\begin{equation*}
\begin{split}
& Z_m(t) = e^{ \overline{A}_{mm} h_m } \Tilde{Y}_m(t-h_m) + \int_{t-h_m}^t  e^{ \overline{A}_{mm} (t-s) } \overline{\sigma}_m(s) dW_s \\
& +  \int_0^{t-2 h_m}  e^{ \overline{A}_{mm} (t-s) } \overline{r}_m(s) ds. \\
& +  \int_{t-2 h_m}^t  e^{ \overline{A}_{mm} (t-s) } \overline{r}_m(s) ds,
\end{split}
\end{equation*}
where $\overline{r}_m(t) = \int_{t-h_m}^t \overline{G}_m(t-s) \overline{\sigma}(s) dW_s.$
% We recall that $\overline{r}_i$ is of the form
% $$
% \overline{r}_i(t) = \int_{t-h_m}^t \overline{G}_i(t-s) \overline{\sigma}(s) dW_s.
% $$
Adjusting the computations from the proof of \cite[Lemma 4]{velho_stabilization_2024}, we obtain %rewrite the last integral in $\overline{r}_m$ as
\small
\begin{equation*}
\begin{split}
& Z_m(t) = e^{ \overline{A}_{mm} h_m } Y_m(t-h_m) + \int_{t-h_m}^t  e^{ \overline{A}_{mm} (t-s) } \overline{\sigma}_m(s) dW_s \\
& +  \int_0^{t-2 h_m}  e^{ \overline{A}_{mm} (t-s) } \overline{r}_m(s) ds. \\
& +  \int_{t-3 h_m}^t \left( \int_{\max(t-2h_m,\tau) }^{\min(t, \tau +h_m)} e^{ \overline{A}_{mm} (t-s) } \overline{G}_m(s-\tau) ds \right) \overline{\sigma}(\tau) dW_\tau.
\end{split}
\end{equation*}
\normalsize
This implies 
%[\commentGV{Attention quand on passe de $\sigma_m$ à $\sigma$ ! $\Gamma$ ne multiplie pas toutes les composantes de $\sigma$ de la même façon. On utilise $P_i$ une "projection" pour l'écrire de façon concise (définie dans le théorème) . Développer un peu plus?}]
\begin{equation*}
\begin{split}
& Z_m(t) = e^{ \overline{A}_{mm} h_m } \Tilde{Y}_m(t-h_m) + R_m^{(U)}(t-h_{m+1})\\
& + \int_{t-h_{m+1}}^t \Gamma_m(t-\tau) \overline{\sigma}(\tau) dW_\tau .
\end{split}
\end{equation*}
where
\begin{equation*}
\begin{split}
 \Gamma_m(u) & \triangleq \int_{\max(u-h_{m+1},0) }^{\min(u, h_m)} e^{  \overline{A}_{mm}(u- s') } \overline{G}_m(s') ds' \\ 
& + \indic_{[0,h_m]}(u) \ e^{ \overline{A}_{mm} u } P_m,
\end{split}
\end{equation*}
and
\begin{equation*}
\begin{split}
R_m^{(U)}(t-h_{m+1}) & \triangleq  \int_0^{t-h_{m+1}}  e^{ \overline{A}_{mm} (t-s) } \overline{r}_m(s) ds \\
& + \int_{t-3h_m}^{t-h_{m+1}} \Gamma_m(t-\tau) \overline{\sigma}(\tau) dW_\tau,
\end{split}
\end{equation*}
which is the searched equality.\\

\noindent\textbf{Induction:} Let $i$ be in $\llbracket 1 , m-1 \rrbracket $ and assume equation \eqref{eq:link_substate_cropped_predictor_tilde_Yi} holds for all $j>i$. By taking equality \eqref{eq:appendix_link_artstein_original_substate_Yi} and injecting equation \eqref{eq:link_substate_cropped_predictor_tilde_Yi} we obtain
\small
\begin{equation*}
\begin{split}
& Z_i(t) = e^{ \overline{A}_{ii} h_i } \Tilde{Y}_i(t-h_i) + \int_{t-h_i}^t  e^{ \overline{A}_{ii} (t-s) } \overline{\sigma}_i(s) dW_s \\
& +  \int_{0}^t  e^{ \overline{A}_{ii} (t-s) } \left( \sum_{j=i+1}^m \overline{A}_{ij}  U_j(s-h_j) \right) ds \\
& +  \int_{0}^t  e^{ \overline{A}_{ii} (t-s) } \overline{r}_i(s) ds \\
& +  \int_{0}^t  e^{ \overline{A}_{ii} (t-s) }\Biggl( \sum_{j=i+1}^m \overline{A}_{ij} \biggl[  e^{ \overline{A}_{jj} h_j } Y_i(s-h_j) \\
& \hspace{12em} + R^{(U)}_j(s-h_{j+1}) \biggr] \Biggr) ds \\
& + \int_{0}^{t-h_{i+1}}  e^{ \overline{A}_{ii} (t-s) } \left( \sum_{j=i+1}^m \overline{A}_{ij}  \int_{s-h_{j+1}}^{s} \Gamma_j(s-\tau) \overline{\sigma}(\tau) dW_\tau \right) ds \\
& +  \int_{t-h_{i+1}}^t  e^{ \overline{A}_{ii} (t-s) } \left( \sum_{j=i+1}^m \overline{A}_{ij}  \int_{s-h_{j+1}}^{s} \Gamma_j(s-\tau) \overline{\sigma}(\tau) dW_\tau \right) ds
\end{split}
\end{equation*}
\normalsize
%\modifJA{Faire rentrer dans la colonne}
Let us focus on the last double integral. Using the stochastic Fubini's theorem we obtain 
\small
\begin{equation*}
\begin{split}
& \int_{t-h_{i+1}}^t  e^{ \overline{A}_{ii} (t-s) } \left( \sum_{j=i+1}^m \overline{A}_{ij}  \int_{s-h_{j+1}}^{s} \Gamma_j(s-\tau) \overline{\sigma}(\tau) dW_\tau \right) ds \\
& = \sum_{j=i+1}^m \int_{t-h_{i+1}-h_{j+1}}^t \\
& \hspace{4em} \left( \int_{\max(t-h_{i+1},\tau )}^{\min(t,\tau+h_{j+1})}  e^{ \overline{A}_{ii} (t-s) }  \overline{A}_{ij}  \Gamma_j(s-\tau)  ds  \right) \overline{\sigma}(\tau) dW_\tau.
\end{split}
\end{equation*}
\normalsize
Finally, using the change of variable $s' = s - \tau$, we have
\begin{equation*}
\begin{split}
& \sum_{j=i+1}^m \int_{t-h_{i+1}-h_{j+1}}^t \\
& \hspace{2em}\left( \int_{\max(t-h_{i+1},\tau )}^{\min(t,\tau+h_{j+1})}  e^{ \overline{A}_{ii} (t-s) }  \overline{A}_{ij}  \Gamma_j(s-\tau)  ds  \right) \overline{\sigma}(\tau) dW_\tau \\
& = \sum_{j=i+1}^m \int_{t-h_{i+1}-h_{j+1}}^t \Theta_{i,j}(t-\tau) \overline{\sigma}(\tau) dW_\tau.
\end{split}
\end{equation*}
with
$$
\Theta_{i,j}(u) \triangleq \int_{\max(u-h_{i+1},0 )}^{\min(u,h_{j+1})}  e^{ \overline{A}_{ii} (u-s') }  \overline{A}_{ij}  \Gamma_j(s')  ds' 
$$
By writing the integral $\int_{t-h_i}^t  e^{ \overline{A}_{ii} (t-s) } \overline{r}_i(s) ds$ as a stochastic integral (as done in the initialization), we obtain
\begin{equation*}
\begin{split}
Z_i(t) & = e^{ \overline{A}_{ii} h_i } \Tilde{Y}_i(t-h_i) + R^{(U)}_i(t-h_{i+1}) \\
    & + \int_{t-h_{i+1}}^t \Gamma_i(t-s) \overline{\sigma}(s) dW_s,
\end{split}
\end{equation*}
with
\small
\begin{equation*}
\begin{split}
& R^{(U)}_i(t-h_{i+1}) \triangleq \sum_{j=i+1}^m \int_{0}^t  e^{ \overline{A}_{ii} (t-s) } \overline{A}_{ij}  U_j(s-h_j) ds \\
& + \int_{t-h_i-h_m}^{t-h_{i+1}} \left( \int_{\max(t-h_{i+1},\tau) }^{\min(t, \tau +h_m)} e^{ \overline{A}_{ii} (t-s) } \overline{G}_i(s-\tau) ds \right) \overline{\sigma}(\tau) dW_\tau   \\
& +  \int_{0}^{t-h_{i+1}}  e^{ \overline{A}_{ii} (t-s) } \overline{r}_i(s) ds \\
& + \sum_{j=i+1}^m \int_{0}^t  e^{ \overline{A}_{ii} (t-s) } \overline{A}_{ij} \left[  e^{ \overline{A}_{jj} h_j } Y_i(s-h_j) + R^{(U)}_j(s-h_{j+1}) \right] ds \\
& + \sum_{j=i+1}^m \int_{t-h_i-h_{j+1}}^{t-h_{i+1}} \Theta_{i,j}(t-\tau) \overline{\sigma}(\tau) dW_\tau.
\end{split}
\end{equation*}
\normalsize
and
\begin{equation*}
\begin{split}
 \Gamma_i(u) & \triangleq \int_{\max(u-h_{i+1},0) }^{\min(u, h_m)} e^{ \overline{A}_{ii} (u-s) } \overline{G}_i(u) ds \\
& + \indic_{[0,h_i]}(u) \ e^{ \overline{A}_{ii} u } P_i + \sum_{j=i+1}^m \Theta_{i,j}(u),
\end{split}
\end{equation*}
which concludes the proof.

\end{proof}

\subsection{Proof of Theorem \ref{thm:controllability_above_minimal_covariance_bound}}\label{APP_THM_CONTROLABILITY}
\begin{proof}
Starting from time $t = T - h_{i+1} $, $\Tilde{Y}_i(t)$ can compensate $R_i^{(U)}(T-h_{i+1})$ as it becomes $\mathcal{F}_{t}$ measurable. We then apply the same reasoning as in \cite[section II.C]{velho_mean-covariance_2023} to get rid of the deterministic drift. We use the control
$$
U_i(t) = U_{i,Y}(t) + U_{i,R}(t),
$$
with
\begin{align*}
&U_{i,R}(t)  = - \Tilde{B}_i^T e^{\overline{A}^T_{ii}(T-h_i-t)} \Bigl( G_{T-h_{i+1}}^{T-h_i} \Bigr)^{-1} \\
&  \times  \Bigl( R_i^{(U)}(T-h_{i+1}) + e^{\overline{A}_{ii}(h_{i+1}-h_i)} \Tilde{Y}(T-h_{i+1}) \Bigr),
\end{align*}
where $G_{T-h_{i+1}}^{T-h_i} \triangleq \int_{T-h_{i+1}}^{T-h_i} e^{\overline{A}_{ii}(T-h_i-t)} \Tilde{B}_i \Tilde{B}_i^T e^{\overline{A}^T_{ii}(T-h_i-t)} dt$ is the grammian of the system, which is invertible by controllability of the matrices $(\overline{A}_{ii}, \Tilde{B}_i)$ \cite{velho_mean-covariance_2023}. This control is $\mathcal{F}_t$-measurable for $t \in [T-h_{i+1}, T-h_{i}]$. We then obtain get
\begin{align*}
Y(T-h_i) & = e^{\overline{A}_{ii}(h_{i+1}-h_i)} \Tilde{Y}(T-h_{i+1}) \\
& + \int_{T-h_{i+1}}^{T-h_i} e^{\overline{A}_{ii}(T-h_i-t)} \sigma(t) dW_t \\
& +  \int_{T-h_{i+1}}^{T-h_i} e^{\overline{A}_{ii}(T-h_i-t)} (U_{i,Y}(t) + U_{i,R}(t) ) dt,
\end{align*}
which yields
$$Z_i(T) = e^{\overline{A}_{ii} h_i} \hat{Y}_i(T-h_i) +  \int_{T-h_{i+1}}^T \Gamma_i(T-s) \overline{\sigma}(s) dW_s,$$
with $\hat{Y}_i$ following the dynamic
$$
\left\{
    \begin{array}{ll}
    d\hat{Y}_i(t) &= \hspace{1em} \left( \overline{A}_{ii} \hat{Y}_i(t) + \Tilde{B}_i U_{i,Y}(t) \right) dt + \overline{\sigma}_i(t)  dW_t, \\
    \hat{Y}_i(0) &= \hspace{1em} 0.
    \end{array}
    \right.
$$
This means we effectively canceled the term $ R_i^{(U)}(T-h_{i+1})$ as well as the initial condition. The same can be done with the stochastic integral $ \int_{T-h_{i+1}}^{T-h_i-\eta} \Gamma_i(T-s) \overline{\sigma}(s) dW_s$ for any $\eta > 0$. The rest of the proof directly follows from classical theorems of stochastic controllability \cite[Theorem 13]{mahmudov_controllability_2001} that gives us the controllability in mean and variance of $\hat{Y}_i$ .
\end{proof}

\subsection{Proof of Theorem \ref{thm:stabilization_system_feedback}}\label{APP_THM_STABILIZATION}
\begin{proof}
The stabilization of $(X, u ,v )$ in the sense of \eqref{eq:stabilization_mean_expectation} follows directly from the stabilization of $(Z, \alpha , \beta )$ as the changes of variables are linear and bounded reversible. We therefore first focus on showing by induction that, for $U_i = - K_i Y_i(t)$ as defined in the Theorem \ref{thm:stabilization_system_feedback}, the following result holds for $t> h_m$, $1 \leq i \leq m$
\begin{equation}\label{eq:appendix_induction_hypo_stabilization_Z}
\Vert \espE[ Z_i(t) ] \Vert \leq C e^{- \nu t} \Vert Y(0) \Vert, \quad  \espE \left[ \Vert Z_i(t) \Vert^2 \right] \leq C  
\end{equation}
where $\nu$ can be prescribed through the gains $K_i$ and $C$ is a possibly overloaded positive constant. The stabilization of $\alpha$ and $\beta$ follow as they can be written explicitly in terms of the control $U_i$ as in \eqref{eq:induction_hypo_lemma_tracking_delayed_SDE_beta}.

\textbf{Initialization:} The initialization can be directly deducted from the scalar one delay case in \cite{velho_stabilization_2024}[Theorem 3].

\textbf{Induction: } Let $1 \leq i \leq m-1$, we suppose that \eqref{eq:appendix_induction_hypo_stabilization_Z} is true for all $j > i$. With the controller $U_i(t) = - K_i Y_i(t)$, by denoting $H_i \triangleq A_{ii} - \Tilde{B}_{ii} K_i$, we can rewrite $Z_i(t)$ as
\begin{align*}
& Z_i(t) = e^{ \overline{A}_{ii} h_i } Y_i(t-h_i) + \int_{t-h_i}^t  e^{ \overline{A}_{ii} (t-s) } \overline{\sigma}_i(s) dW_s \\
& + \int_{t-h_i}^t e^{ \overline{A}_{ii} (t-s) } \overline{r}_i(s) ds \\
& + \int_{t-h_i}^t e^{ \overline{A}_{ii} (t-s) } \left( \sum_{j > i} \overline{A}_{ij} Z_j(s) + \sum_{j > i} \overline{B}_{ij} U_j(s-h_j) \right)ds,
\end{align*}
with 
\small
\begin{align*}
& Y_i(t) = e^{H_i t } Y_i(0) + \int_{0}^t  e^{ H_{i} (t-s) } \overline{\sigma}_i(s) dW_s \\
& + \int_{0}^t e^{ H_{i} (t-s) } \left( \sum_{j > i} \overline{A}_{ij} Z_j(s) + \sum_{j > i} \overline{B}_{ij} U_j(s-h_j) + \overline{r}_i(s) \right)ds
\end{align*}
\normalsize
We recall that the mean of a stochastic integral is null and therefore the mean of $r(t)$ is zero. We can deduce from this and the induction hypothesis \eqref{eq:appendix_induction_hypo_stabilization_Z} that
\begin{align*}
& \Vert \espE [ Z_i(t) ] \Vert \leq C \Vert \espE [ Y_i(t-h_i) ] \Vert + C e^{- \nu t}\Vert Y(0) \Vert.
\end{align*}
With the same arguments combined with the expression of $Y$ and the fact that $H_i$ is Hurwitz, we have that
\begin{align*}
& \Vert \espE [ Y_i(t) ] \Vert \leq C e^{- \nu t}\Vert Y(0) \Vert.
\end{align*}
Similarly, since the variance can be bounded by the sum of the variances (up to a constant), we have that 
\begin{align*}
& \espE \left[ \Vert Z_i(t) \Vert^2 \right] \leq C \espE \left[ \Vert Y_i(t-h_i) \Vert^2 \right] + C.
\end{align*}
Since the drift in the dynamic of $Y_i$, composed of the terms $r_i$, $U_j$ and $Z_j$ for $j>i$, has a bounded variance, we can use similar arguments as in the scalar one delay case to bound the variance of the predictor $Y_i$ with the feedback controller. More details can be found in the proof of \cite{velho_stabilization_2024}[Theorem 3].
%\modifJA{A completer}
\end{proof}

\subsection{Proof of Theorem \ref{thm:sequence_of_controls_to_minimal_variance}}\label{APP_THM_SEQUENCE_TO_MINIMAL}
\begin{proof}
To obtain the needed sequence of controls, we prove the following lemma by induction

\begin{lemma}\label{lem:lemma_appendix_high_gain_controller}
    Let $(K_i)_{1 \leq i \leq m}$ be $m$ positive real numbers. There exists an adapted controller $U$ such that the state $Z$ associated verifies for all $1 \leq i \leq m$
    \small
    \begin{equation}\label{eq:high_gain_controller_induction_lemma_appendix}
    \begin{split}
        Z_i(t) & = e^{\overline{A}_{ii} h_i } \left( e^{- K_i (t-h_i)} Y_i(0) + \int_0^{t-h_i} e^{-K_i(t-h_i-s)} \Hat{\sigma}_i(s) dW_s \right) \\
        & + \int_{t-h_i}^t \Gamma_i(t-s) \overline{\sigma}(s) dW_s
    \end{split}
    \end{equation}
    \normalsize
    with $\Hat{\sigma}_i \in L^\infty([0,T] ; \mathbb{R})$.
\end{lemma}
We can then prove that as the gains $K_i$ go to infinity, the weighted variance of the state goes to the minimal variance.
\begin{proof}[Proof of Lemma \ref{lem:lemma_appendix_high_gain_controller}]

\ \\ \noindent\textbf{Initialization:} We have
\small
\begin{equation*}
\begin{split}
& Z_m(t) = e^{ \overline{A}_{mm} h_m } Y_m(t-h_m) + \int_{t-h_m}^t  e^{ \overline{A}_{mm} (t-s) } \overline{\sigma}_m(s) dW_s \\
& +  \int_{t-2 h_m}^t \left( \int_{\max(t-h_m,\tau) }^{\min(t, \tau +h_m)} e^{ \overline{A}_{mm} (t-s) } \overline{G}_m(s-\tau) ds \right) \overline{\sigma}(\tau) dW_\tau .
\end{split}
\end{equation*}
\normalsize
Defining
$$
g_m(u) \triangleq e^{- \overline{A}_{mm} h_m} \int_{u}^{h_m} e^{ \overline{A}_{mm} (u-s') } \overline{G}_m(s') ds',
$$
and
%\small
\begin{equation*}
\begin{split}
& \overline{Y}_m(t) \triangleq  Y_m(t)  + \int_{t- h_m}^t g_m(t-s) \overline{\sigma}(s) dW_s, \\
%& + e^{ - \overline{A}_{mm} h_m } \int_{t-2 h_m}^{t-h_m} \left( \int_{\max(t-h_m,\tau) }^{\min(t, \tau +h_m)} e^{ \overline{A}_{mm} (t-s) } \overline{G}_m(s-\tau) ds \right) \overline{\sigma}(\tau) dW_\tau
%& + e^{ - \overline{A}_{mm} h_m } \int_{t-2 h_m}^{t-h_m} \left( \int_{t-h_m }^{ \tau +h_m} e^{ \overline{A}_{mm} (t-s) } \overline{G}_m(s-\tau) ds \right) \overline{\sigma}(\tau) dW_\tau
\end{split}
\end{equation*}
%\normalsize
we obtain
$$
Z_m(t) = e^{ \overline{A}_{mm} h_m } \overline{Y}_m(t-h_m) + \int_{t-h_m}^t \Gamma_m(t-s) \overline{\sigma}(s) dW_s.
$$
%which means that our goal is now to find a sequence of controls $U_m^{(n)}$ such that $\overline{Y}_m^{(n)}$ has a variance as small as possible. 
Adjusting the proof of~\cite[Lemma 6]{velho_stabilization_2024} we have that $\overline{Y}_m$ verifies\begin{equation}\label{eq:SDE_overline_Y_change_var_artstein}
\left\{
    \begin{array}{ll}
    d\overline{Y}_m(t) &= \hspace{1em} \left( \overline{A}_{mm} \overline{Y}_m(t) + \Tilde{B}_m U_m(t) + \Tilde{r}_m(t)   \right) dt \\
    & +  (\overline{\sigma}_m(t) + g_m(0) \overline{\sigma} (t) ) dW_t, \\
    \overline{Y}_m(0) &= \hspace{1em} Y_m(0) ,
    \end{array}
    \right.
\end{equation}
with $\Tilde{r}_m(t) \triangleq  \Tilde{g}_m(t-s) \overline{\sigma}(s) dW_s $ and
$$
\Tilde{g}_m(u) \triangleq \overline{G}_m(u) + g_m'(u) - \overline{A}_{mm} g_m(u) .
$$
We can now set our controller to be 
$$
U_m(t) = - \Tilde{B}_m^{-1} (K_m + \overline{A}_{mm}) \overline{Y}_m(t) - \Tilde{B}_m^{-1} \Tilde{r}_m(t),
$$
and obtain
\small
$$
\overline{Y}_m(t) = e^{- K_m t } Y_m(0) + \int_0^t e^{- K_m (t-s) } (\overline{\sigma}_m(s) + g_m(0) \overline{\sigma} (s) ) dW_s,
$$
\normalsize
which yields equation \eqref{eq:high_gain_controller_induction_lemma_appendix} for $i = m$. 
%we can prove the integral of the weighted variance goes to 0 as the gain $K_m$ gets bigger. %\modifJA{Ref?}
%\commentGV{Initialization a rendre cohérente avec la preuve du Theoreme \ref{thm:substate_separation_for_min_cov_induction}}

\textbf{Induction:}

Let $i$ such that $1 \leq i \leq m-1$. We recall that
\begin{equation*}
\begin{split}
dZ_i(t) & = \left( \overline{A}_{ii} Z_i(t) + \overline{B}_{ii} U_i(t-h_i) \right)dt \\
& + \left( \sum_{j = i+1}^m \overline{A}_{ij} Z_j(t) + \sum_{j = i+1}^m \overline{B}_{ij} U_j(t-h_j) \right)dt \\
& + \overline{r}_i(t) dt + \overline{\sigma}_i(t) dW_t.
\end{split}
\end{equation*}
we can use the induction hypothesis \eqref{eq:high_gain_controller_induction_lemma_appendix} and inject the formula for $Z_j$ and obtain
\small
\begin{equation*}
\begin{split}
& dZ_i(t) = \left( \overline{A}_{ii} Z_i(t) + \overline{B}_{ii} U_i(t-h_i) + \sum_{j = i+1}^m \overline{B}_{ij} U_j(t-h_j) \right)dt \\
& +  \sum_{j = i+1}^m \overline{A}_{ij} \biggl(  e^{\overline{A}_{jj} h_j } e^{- K_j (t-h_j)} Y_j(0) \\ 
& \hspace{0.5em} + \int_0^{t-h_j} e^{-K_j(t-h_j-s)} \Hat{\sigma}_j(s) dW_s + \int_{t-h_j}^t \Gamma_j(t-s) \overline{\sigma}(s) dW_s            \biggr) \\
 & \hspace{0em}  + \overline{r}_i(t) dt + \overline{\sigma}_i(t) dW_t
\end{split}
\end{equation*}
\normalsize
We can cancel out most of the terms with the controller
\small
\begin{equation*}
\begin{split}
& U_i(t-h_i) = \Hat{U}_i(t-h_i) \\
& - \overline{B}_{ii}^{-1}  \sum_{j = i+1}^m  \biggl(\overline{B}_{ij} U_j(t-h_j) + \overline{A}_{ij} e^{\overline{A}_{jj} h_j } e^{- K_j (t-h_j)} Y_j(0) \\
%&   \quad   + \int_0^{t-h_j} e^{-K_j(t-h_j-s)} \Hat{\sigma}_j(s) dW_s + \int_{t-h_j}^{t-h_i} \Gamma_j(t-s) \overline{\sigma}(s) dW_s   \biggr)
&   \hspace{6em}    + \int_0^{t-h_j} e^{-K_j(t-h_j-s)} \Hat{\sigma}_j(s) dW_s   \biggr)
\end{split}
\end{equation*}
\normalsize
which gives the following dynamic for $Z_i$
\begin{equation*}
\begin{split}
& dZ_i(t)  = \left( \overline{A}_{ii} Z_i(t) + \overline{B}_{ii} \Hat{U}_i(t-h_i) + \overline{r}_i(t) \right)dt \\ 
& \hspace{1em} + \left( \sum_{j = i+1}^m \int_{t-h_j}^t \Gamma_j(t-s) \overline{\sigma}(s) dW_s\right)dt  + \overline{\sigma}_i(t) dW_t.
\end{split}
\end{equation*}
Using the Artstein predictor $\Hat{Y}_i$ with the control $\Hat{U}_i$, and the computations of the proof of Theorem \ref{thm:substate_separation_for_min_cov_induction} we can show that, there exists a deterministic $C^1$ function $g_i$ such that 
\begin{equation*}
\begin{split}
& Z_i(t)  =  e^{ \overline{A}_{ii} h_i } \left( \Hat{Y}_i(t-h_i) + \int_{t- 2 h_m}^{t-h_i} g_i(t-s) dW_s \right) \\
    & + \int_{t-h_i}^t \Gamma_i(t-s) \overline{\sigma}(s) dW_s
\end{split}
\end{equation*}
We can now use the same reasoning as in the initialization, defining 
$$
\overline{Y}_i = \Hat{Y}_i(t-h_i) + \int_{t- 2 h_m}^{t-h_i} g_i(t-s) dW_s
$$
and the associated control 
$$
\Hat{U}_i(t) = - \Tilde{B}_i^{-1} (K_i + \overline{A}_{ii}) \overline{Y}_i(t) - \Tilde{B}_i^{-1} \Tilde{r}_i(t),
$$
so that $\overline{Y}_i$ verifies
$$
\overline{Y}_i(t) = e^{- K_i t } Y_i(0) + \int_0^t e^{- K_i (t-s) } (\overline{\sigma}_i(s) + g_i(0) \overline{\sigma} (s) ) dW_s,
$$
which yields equation \eqref{eq:high_gain_controller_induction_lemma_appendix} and concludes our proof of the lemma.
%\textcolor{red}{Et donc on a $\Tilde{r}(u) = 0 $!!! Fait du sens quand on réfléchit au principe d'équivalence de controle optimal linéaire stochastique et déterministe... }
%\commentGV{Le retour d'état revient donc à effectivement augmenter la diffusion par un facteur. Comment l'expliquer le mieux? Soit changement de variable dès le début, soit on le garde pour la fin. La dérivation est assez sketchy, il faut la vérifier également}
\end{proof}

We can now show that for all $t$, the variance of $Z_i$ will go to the minimal variance as $K_i$ goes to $+ \infty$. In equation \eqref{eq:high_gain_controller_induction_lemma_appendix}, we note that the three terms are independent. This means that we the weighted variance of $Z_i$ can be separated as follows
\begin{equation*}
\begin{split}
& \espE[Z_i(t) Q_i(t) Z_i(t)] =   Q_i(t) \left( e^{\overline{A}_{ii} h_i } e^{- K_i (t-h_i)} Y_i(0) \right)^2 \\
& \quad +  Q_i(t) \int_0^{t-h_i} \left( e^{\overline{A}_{ii} h_i } e^{-K_i(t-h_i-s)} \Hat{\sigma}_i(s) \right)^2 ds \\
& \quad +  Q_i(t) \int_{t-h_i}^t \left(  \Gamma_i(t-s) \overline{\sigma}(s) \right)^2 ds.
\end{split}
\end{equation*}
We also note that the last integral is equal to the minimal weighted variance $V^{(i)}_{\min,Q}(t)$. Denoting $C$ a possibly overloaded positive constant, the integral of the variance can be bounded with the Itô isometry by
\begin{equation*}
\begin{split}
& \espE \left[ \int_{h_m}^T Z_i(t) Q_i(t) Z_i(t) \right] \leq   C \Vert Q \Vert_\infty | Y_i(0) |^2 \int_{h_m}^T e^{- 2 K_i t} dt\\
& \quad +  C \Vert Q \Vert_\infty \Vert \Hat{\sigma} \Vert_\infty^2  \int_{h_m}^T \left( \int_0^{t-h_i} e^{-2 K_i(t-h_i-s)} ds \right) dt\\
& \quad +  J_{\min,Q}^{(i)} \\
& \quad \leq   C \Vert Q \Vert_\infty | Y_i(0) |^2  \frac{e^{- 2 K_i h_m}}{K_i} + C \Vert Q \Vert_\infty \Vert \Hat{\sigma} \Vert_\infty^2 \frac{T}{K_i} +  J_{\min,Q}^{(i)}.
\end{split}
\end{equation*}
Therefore, if we chose a sequence of controls $U^{(n)}$ such that $K_i^{(n)}$ goes to zero, the associated state $Z_i^{(n)}$ weighted variance goes to $J_{\min,Q}^{(i)}$.

\end{proof}

\section{Supplementary materials}\label{A2}

\ 
Appendix for additional details on certain technical computations. For Authors and Reviewers only.
\

\subsection{Exchanging integrals in the proof of Lemma }\label{EXCHANGE_INT}

Let us define 
\small
\begin{align*}
    I_i & \triangleq \int_{t - \frac{1-x}{\mu_i}}^t \Biggl( \sum_{j=i+1}^m \omega_{i,j}(x + \mu_i(t-s)) \times \\
& \hspace{3em} \int_{s - \frac{1-x - \mu_i(t-s)}{\mu_j}}^s G_j \bigl( x + \mu_i(t-s) , s-\tau \bigr) \sigma(\tau) dW_\tau  \Biggr) ds 
\end{align*}
\normalsize
We want to use the stochastic Fubini theorem and write it as 
$$
I_i = \int_{t - \frac{1-x}{\mu_i}}^t G_i(x,t-\tau) \sigma(\tau) dW_\tau
$$
with $G_i$ a deterministic function.
We first note that in the stochastic integral, $\tau$ can take values in $[t-\frac{1-x}{\mu_i}, t ]$, we can therefore rewrite the integral as
\begin{align*}
    I_i & = \sum_{j=i+1}^m \int_{t - \frac{1-x}{\mu_i}}^t \Biggl( \int_{t - \frac{1-x}{\mu_i}}^t \indic_{\mathcal{S}_{ij}}(s,\tau) \omega_{i,j}(x + \mu_i(t-s)) \times \\
& \hspace{8em}  G_j \bigl( x + \mu_i(t-s) , s-\tau \bigr) \sigma(\tau) dW_\tau  \Biggr)  ds 
\end{align*}
where $\mathcal{S}_{ij} \subset \mathbb{R}^2$ is given by  
\begin{align*}
\mathcal{S}_{ij} & \triangleq \biggl \{ (s,\tau)  \  \vert \  s \in \left[t-\frac{1-x}{\mu_i}, t \right], \\
& \hspace{5em} \tau \in \left[s - \frac{1-x - \mu_i(t-s)}{\mu_j}, s \right]  \biggr \}.
\end{align*}
% see Figure \ref{fig:drawing_of_domain_S}.
% \begin{figure}[ht!]
% \centering
%   \includegraphics[width=0.9\linewidth]{images/Dessin_domaine_integration_Gi.png}
%   \caption{\centering Domain $\mathcal{S}_{ij}$.}% \textcolor{red}{mettre LQ et non pas LQT}}
%   \label{fig:drawing_of_domain_S}
% \end{figure}
The bounds on $\tau$ are given by two affine functions of $s$,% as highlighted in Figure \ref{fig:drawing_of_domain_S}, 
which we can reverse to obtain the bounds of $s$ in terms of $\tau$. 
The bounds are given by
$$
\underline{\tau}(s) = \left(  1 - \frac{\mu_i}{\mu_j}  \right) s - \frac{1-x - \mu_i t}{\mu_j}
$$
and
$$
\overline{\tau}(s) = s
$$
The inverse is given by
$$
\underline{s}(\tau) = \tau
$$
and 
\begin{align*}
\overline{s}(\tau) & = \left( \tau + \frac{1-x - \mu_i t}{\mu_j} \right)  \left(  1 - \frac{\mu_i}{\mu_j}  \right)^{-1} \\
 & = \tau + \frac{1-x - \mu_i(t-\tau)}{\mu_j - \mu_i}.
\end{align*}
By taking into account that $s \leq t$, we can therefore write $\mathcal{S}_{ij}$ as 
\begin{align*}
\mathcal{S}_{ij} &  \triangleq \biggl \{ (s,\tau)  \  \vert \  \tau \in \left[t-\frac{1-x}{\mu_i}, t \right],  \\
& \hspace{4em} s \in \left[\tau, \min \left( \tau + \frac{1-x - \mu_i(t-\tau)}{\mu_j - \mu_i} , t \right) \right]  \biggr \}.
\end{align*}
Therefore, by setting $\overline{s}_{i,j}(x,u) \triangleq \min \left(  \frac{1-x-\mu_i u}{\mu_j - \mu_i} , u\right)$ and using the stochastic Fubini theorem %[\commentGV{detail assumptions?}]
, we have
\begin{align*}
    I_i & = \sum_{j=i+1}^m \int_{t - \frac{1-x}{\mu_i}}^t \Biggl( \int_{\tau}^{\tau + \overline{s}_{i,j}(x,t-\tau)} \omega_{i,j}(x + \mu_i(t-s)) \times \\
& \hspace{8em}  G_j \bigl( x + \mu_i(t-s) , s-\tau \bigr)   ds \Biggr)  \sigma(\tau) dW_\tau
\end{align*}
which concludes our proof with the change of variable $s' = \tau + s$ and with the following expression of $G_i$ 
\small
$$
G_i \triangleq \sum_{j=i+1}^m \int_0^{\overline{s}_{i,j}(x,u)} \omega_{i,j}\bigl (x + \mu_i (u-s') \bigr) G_j \bigl (x + \mu_i (u-s'), s' \bigr ) ds'.
$$
\normalsize

\end{document}